\documentclass[11pt]{amsart}

\usepackage{epsfig}

\usepackage{graphicx}
\usepackage{amscd}
\usepackage{amsmath}
\usepackage{amsfonts}
\usepackage{yfonts}
\usepackage{psfrag}
\usepackage{amssymb}
\usepackage{verbatim}
\usepackage{comment}

\newtheorem{theorem}{Theorem}[section]
\theoremstyle{plain}

\newtheorem{corollary}[theorem]{Corollary}

\newtheorem{defi}[theorem]{Definition}

\newtheorem{lemma}[theorem]{Lemma}

\newtheorem{prop}[theorem]{Proposition}
\newtheorem{remark}[theorem]{Remark}

\numberwithin{equation}{section}

\def\Error{{\mathcal R}}
\def\wmu{\widetilde{\mu}}
\def\Jk{{\mathcal J}}
\def\Xxi{{\mathfrak X}_\zeta}
\def\Xx{{\mathfrak X}}
\def\FrA{{\mathfrak A}}
\def\One{{1\!\!1}}

\def\dist{{\rm dist}}

\def\Hk{{\mathcal H}}

\def\wtil{\widetilde}
\def\half{\frac{1}{2}}

\newcommand{\lam}{\lambda}

\newcommand{\gam}{\gamma}
\newcommand{\om}{\omega}

\def\Om{\Omega}
\newcommand{\Sig}{\Sigma}
\def\Sigb{{\mathbf \Sig}}

\newcommand{\Gam}{\Gamma}
\newcommand{\sig}{\sigma}
\def\bbe{\vec{e}}
\newcommand{\R}{{\mathbb R}}

\newcommand{\Z}{{\mathbb Z}}
\newcommand{\C}{{\mathbb C}}

\def\N{{\mathbb N}}

\newcommand{\Nat}{{\mathbb N}}

\def\A{{\mathcal A}}

\def\Sk{{\mathcal S}}

\def\Pk{{\mathcal P}}

\def\Sf{{\sf S}}
\def\T{{\mathbb T}}

\def\one{\vec{1}}
\def\be{\begin{equation}}
\def\ee{\end{equation}}
\newcommand{\Ek}{{\mathcal E}}
\newcommand{\Fk}{{\mathcal F}}

\newcommand{\eps}{{\varepsilon}}
\newcommand{\es}{\emptyset}

\def\ov{\overline}
\def\und{\underline}

\def\Card{{\rm card}}
\def\card{{\rm card}}

\def\a{a}

\def\ve1{\vec{1}}

\def\M{{\mathbf M}}

\def\Ak{{\mathcal A}}
\def\Pbi{{\mathbf \Pi}}
\def\Bu{B}

\begin{document}

\title[Spectral measures in substitution dynamics]{On the modulus of continuity for spectral measures in substitution dynamics}

\author{Alexander I. Bufetov}
\address{Alexander I. Bufetov\\Aix-Marseille Universit{\'e}, CNRS, Centrale Marseille, I2M, UMR 7373,  Marseille, France; Steklov Institute, Moscow; 
The Institute for Information Transmission Problems, Moscow; 
National Research University Higher School of Economics, Moscow, Russia;
Rice University, Houston TX USA
}
\email{bufetov@mi.ras.ru}
\author{Boris Solomyak }
\address{Boris Solomyak, Box 354350, Department of Mathematics,
University of Washington, Seattle, WA, USA}
\email{solomyak@uw.edu}

\begin{abstract}
The paper gives first quantitative estimates on the modulus of
continuity of the spectral measure  for weak mixing suspension flows over substitution
automorphisms, which yield information about the ``fractal'' structure of these measures. The main results
are, first, a H\"older estimate for the spectral measure of almost all
suspension flows with a piecewise constant roof function; second, a
log-H\"older estimate for self-similar suspension flows; and, third, a
H\"older asymptotic expansion of the spectral measure at zero for such
flows. Our second result implies log-H\"older estimates for the spectral measures of translation flows 
 along stable foliations of pseudo-Anosov automorphisms. 
A key technical tool in the proof of the second result is an ``arithmetic-Diophantine'' proposition, which has other applications.
In the appendix this proposition is used to derive new decay estimates for the Fourier transforms of 
Bernoulli convolutions.
\end{abstract}



\maketitle

\thispagestyle{empty}

\section{Introduction}

Substitution dynamical systems and their spectral properties have been studied for a long time, see \cite{Queff,Siegel} and references therein. These systems are  of intrinsic interest, but also have many links to other areas, both in dynamics and beyond. An incomplete list of these links includes Bratteli-Vershik (adic) transformations \cite{Vershik,VerLiv}, especially in the stationary case; interval exchange transformations that are periodic for the Rauzy-Veech induction, and translation flows along stable/unstable flows for pseudo-Anosov automorphisms, see \cite{Bufetov0,Bufetov1}. Substitutions and associated dynamical systems are also widely used in mathematical physics, in particular, in the study of quasicrystals, see e.g.\ \cite{Aubry,GL,GK,LMS2}.

The aim of this paper is to estimate the modulus of continuity for the
spectral measures of suspension flows over substitution dynamical systems.
Our main assumption is that
the substitution matrix have at least two eigenvalues outside the unit circle, which implies that almost every suspension flow, in particular, the self-similar suspension flow, is weak mixing \cite{CSa,SolTil}.

Our first main result is Theorem~\ref{th-holder1} that gives a uniform H\"older bound
away from zero for the spectral measure of almost all suspension flows
with  piecewise constant roof functions. This result does not, however,
give specific examples of flows with H\"older spectrum. In the important special case
of self-similar suspension flows we are able to obtain log-H\"older
estimates on the spectrum; these are contained in our second main result,
Theorem \ref{th-main2}. Our third main result, Theorem \ref{th-zero}, gives a H\"older
asymptotic expansion for the spectral measure of our self-similar suspension flows {\it
at zero};  the H\"older exponent is explicitly computed.

The motivation is to study fine ``dimension-like,'' or ``fractal'' properties of the spectral measures. H\"older estimates imply lower bounds
for the local dimension and Hausdorff dimension of the measures. In mathematical physics (quantum dynamics, discrete Schr\"odinger
operators, etc.) similar investigations have been very active, see e.g.\ \cite{Last,DG}, 
but we are not aware of any significant work in this
direction in ergodic theory. The hope is that such a study will lead to better understanding of  substitution dynamical systems; in particular, there should be
a connection with ``quantitative rates of weak mixing" (which are conjugacy invariants) in the spirit of \cite{Knill}.

It is worth mentioning that some of our techniques have parallels in the work of Forni and Ulcigrai \cite{FU}, who proved Lebesgue
spectrum of smooth time-changes for the classical horocycle flow on a compact hyperbolic surface using asymptotic properties of ``twisted
ergodic integrals.'' Our setting is completely different, but we are also using such integrals in Sections 4 and 5.

The paper is organized as follows. In Section 2 we represent spectral
measures of substitution automorphisms by matrix Riesz products; the
construction is a generalization of the  Riesz product
representation of the spectrum for substitutions of constant length, see \cite{Queff}.  In
Section 3 we analyze our matrix Riesz products; the main result of this
section is Proposition \ref{prop-Dioph}, an upper bound for the entries of our
matrix products. In the following sections, Proposition \ref{prop-Dioph} is used
to reduce questions about spectral measures of substitutions  to problems
in Diophantine approximation. In Section 4, we formulate and prove Theorem~\ref{th-holder1} on the H\"older property for the spectrum for almost all suspension
flows with piecewise constant roof function; the Hausdorff dimension of
the set of exceptional roof functions is estimated in Theorem~\ref{th-holder2}, whose
proof relies on a generalization of the ``Erd\H{o}s-Kahane argument'' in the theory of Bernoulli convolutions. In
Section 5 we formulate and prove Theorem \ref{th-main2} on the log-H\"older property
for self-similar  suspension flows. 
As a corollary, we obtain  log-H\"older estimates for the spectral measures of translation flows 
 along stable foliations of pseudo-Anosov automorphisms (Corollary \ref{cor-IET}). 
The main step in the argument is
Proposition \ref{prop-alg}, which makes use of techniques from Pisot and Salem.  In Section 6, we give an asymptotics of the spectral measure
at zero using the asymptotics of ergodic integrals expressed
in terms of finitely-additive invariant measures from \cite{Bufetov0,Bufetov1,BuSol}.
In a short Section 7 we collect the implications of our results for the dimension properties of spectral measures.
The Appendix (Section 8) contains some standard proofs that we rely on, and a subsection on Bernoulli convolutions;
in particular,
from Proposition \ref{prop-alg} we derive Corollary \ref{cor-BC} that gives new upper bounds
for Fourier  transforms of Bernoulli convolutions corresponding to
inverses of algebraic integers.

\medskip

{\bf Acknowledgements.}
Work on this project was begun when the authors 
were visiting the Centre International de Rencontres Math{\'e}matiques in Luminy in the framework of 
the ``Recherches en Bin{\^o}me''
programme. We are deeply grateful to the Centre for the warm hospitality.  
A. Bufetov has been supported by the A*MIDEX project (no. ANR-11-IDEX-0001-02) funded by the ``Investissements d'Avenir''
French Government program, managed by the French National Research Agency (ANR).
A. Bufetov has also been supported in part by  the Grant MD-2859.2014.1 of the President of the Russian Federation,
by the Programme ``Dynamical systems and mathematical control theory''
of the Presidium of the Russian Academy of Sciences, by the ANR under the project ``VALET'' of the Programme JCJC SIMI 1,
and by the
RFBR grants 11-01-00654, 12-01-31284, 12-01-33020, 13-01-12449.

 B. Solomyak  is supported in part by NSF grant DMS-0968879. 
He was also partially supported by the Forschheimer Fellowship and the ERC AdG 267259 grant at the Hebrew University of Jerusalem when working on this project.

We are deeply grateful to the referee for many helpful comments.

\section{Matrix Riesz products for substitutions}

Our aim in this section is to represent the spectral measure of a substitution dynamical system via a matrix analog of Riesz products. The reader is referred to \cite{Queff,Siegel} for the background on substitutions and substitution dynamical systems. The main result of this section is Lemma \ref{lem-Riesz}. Various classes of generalized Riesz products have
appeared in the spectral theory of measure-preserving transformations, especially those of {\em rank one}, see e.g.\ \cite{Bourgain1,ChoNa,KleRe,DoEig,Abda}. However, substitution dynamical systems are often of finite rank higher than one, see \cite{Ferenczi}. Transition to higher rank requires us to consider {\it matrix analogues }
of Riesz products.

\subsection{Substitutions.}
For $m\ge 2$ consider the finite alphabet $\A=\{1,\ldots,m\}$ and the set $\A^+$  of nonempty words with letters in $\A$.   
A {\em substitution}  is a map $\zeta:\,\A\to \A^+$, extended to 
 $\A^+$ and $\A^{\N}$ by
concatenation. The {\em substitution space} is defined as the set of bi-infinite sequences $x\in \A^\Z$ such that any word  in $x$
appears as a subword of $\zeta^n(\a)$ for some $\a\in \A$ and $n\in \N$. The {\em substitution dynamical system}  is the left
shift on $\A^\Z$ restricted to $X_\zeta$, which we denote by $T_\zeta$.

The {\em substitution matrix} $\Sf=\Sf_\zeta=(\Sf(i,j))$ is the $m\times m$ matrix, where $m=\# \A$, such that $\Sf(i,j)$ is the number
of symbols $i$ in $\zeta(j)$. The substitution is {\em primitive} if $\Sf_\zeta^n$ is strictly positive (entrywise) for some $n\in \Nat$.
It is well-known that primitive substitution $\Z$-actions are minimal and uniquely ergodic.
We assume that the substitution is primitive and {\em non-periodic}, which in the primitive case is equivalent to the space $X_\zeta$ being infinite.

The length of a word $u$ is denoted by $|u|$. The substitution $\zeta$ is said to be of {\em constant length} $q$ if $|\zeta(a)|=q$ for all $a\in \A$, otherwise, it is of {\em non-constant length}.
Spectral properties of substitution dynamical systems have been studied extensively, see \cite{Queff,Siegel} and references therein. These systems are never strongly mixing \cite{DK}, hence there is always a singular spectral component. The discrete part of the spectrum is well-understood. For substitutions of constant length $q$ the group of eigenvalues is non-trivial and contains
the group of $q$-adic rationals. Non-constant length substitutions may be weakly mixing; this depends on  algebraic and number-theoretic properties of the substitution matrix, see \cite{Host,FMN}.
Here we focus on weakly mixing substitutions (although the results of Sections 2 and 3 apply to those with a discrete spectral component as well). It is an open problem to decide when the spectrum is purely singular, as well as to determine finer properties of the spectral measures, such as their modulus of continuity. To our knowledge, results in this
direction have only been available in the constant length case. For instance, pure singular spectrum is known for the Thue-Morse substitution \cite{Kaku} and its ``abelian'' generalizations \cite{Queff,Baake}.
On the other hand, there are substitutions of constant length with a Lebesgue spectral component, such as the  Rudin-Shapiro substitution, see \cite{Queff}, and its generalizations, due to N. P. Frank \cite{Frank}.

\subsection{Spectral measures.}
Let $(X,T,\mu)$ be a measure-preserving transformation and let $U= U_T:\ f\mapsto f\circ T$ be the Koopman operator on $L^2(X,\mu)$.
Recall that for $f,g\in L^2(X,\mu)$ the (complex) spectral measure $\sig_{f,g}$ is determined by the equations
$$\widehat{\sig}_{f,g}(-k) =\int_0^1 e^{2\pi i k \om}\,d\sig_{f,g}(\om)=
\langle U^k f,g\rangle,\ \ k\in \Z.
$$
We write $\sig_f = \sig_{f,f}$.
The following is standard; see Appendix for the proof.

\begin{lemma}\label{lem-spec1}
For any $f,g\in L^2(X,\mu)$ we have
$$
\sig_{f,g}= \mbox{\rm weak*-}\lim_{N\to \infty} \frac{1}{N} \left\langle \sum_{n=0}^{N-1} e^{-2 \pi i n \om}U^n f, \sum_{n=0}^{N-1} e^{-2 \pi i n \om}U^n g\right\rangle\,d\om,
$$
where in the right-hand side we consider the weak*-limit of absolutely continuous measures with the given density.
\end{lemma}

If $X$ is a metric space and $(X,T,\mu)$ is uniquely ergodic, then for all $f,g\in C(X)$ and all $k\in \Z$:
\begin{eqnarray}
\widehat{\sig}_{f,g}(-k)=\langle U^k f,g\rangle & = & \int_X f(T^k x) \ov{g(x)}\,d\mu(x) \nonumber \\
                                                                                                    & = & \lim_{N\to \infty} \frac{1}{N} \sum_{n=0}^{N-1} 
                                                                                                   f(T^{n+k} x)\ov{g(T^n x)} \label{unerg1},
\end{eqnarray}
where $x\in X$ is arbitrary and the limit is uniform in $x$.

\subsection{Correlation measures.} For substitutions the most important spectral measures are
$$
\sig_\a:= \sig_{\One_{[\a]}}\ \ \ \mbox{and}\ \ \ \sig_{a,b}:= \sig_{\One_{[\a]},\One_{[b]}}.
$$
They are also known as {\em correlation measures}. 
In view of (\ref{unerg1}), since primitive substitution dynamical systems are uniquely ergodic,
\be \label{corr1}
\widehat{\sig_{a,b}}(-k) = \lim_{N\to \infty} \frac{1}{|\zeta^N(\gam)|} \,\Card\left\{0 \le n+k < |\zeta^N(\gam)|:\ \zeta^N(\gam)_{n+k} = \a,\ \zeta^N(\gam)_n = b\right\}
\ee
for any $\gam\in \A$. For a word $v= v_0 v_1\ldots\in \A^+$ let
\be \label{def-Phi}
\Phi_\a(v,\om) = \sum_{j=0}^{|v|-1} \delta_{v_j,\a} e^{-2\pi i  \om j},
\ee
where $\delta_{v_j,a}$ is Kronecker $\delta$  (one if $v_j=a$ and zero otherwise).
Then it is immediate from (\ref{corr1}), as in Lemma~\ref{lem-spec1}, that
\be \label{corr2}
\sig_{a,b} = \mbox{\rm weak*-}\lim_{N\to \infty} \frac{1}{|\zeta^N(\gam)|}\,{\Phi_a(\zeta^N(\gam),\om)}\cdot\ov{\Phi_b(\zeta^N(\gam),\om)}\,d\om
\ee
for any $\gam\in \A$.

Observe that for any two words $u,v$ and the concatenated word $uv$ we have
 \be \label{eq-Phi}
 \Phi_\a(uv,\om) = \Phi_\a(u,\om) + e^{-2\pi i \om |u|} \Phi_\a(v,\om).
 \ee
Suppose $\zeta(b) = u_1^{(b)} \ldots u_{k_b}^{(b)}$ for $b\in \A$. Then $\zeta^{n}(b)= \zeta^{n-1}(u_1^{(b)})\ldots \zeta^{n-1}(u_{k_b}^{(b)})$ for $n\ge 1$, hence (\ref{eq-Phi}) implies for all $b\in \A$:
$$
\Phi_a(\zeta^{n}(b),\om) = \sum_{j=1}^{k_b} \exp\left[-2\pi i \om \left(|\zeta^{n-1}(u_1^{(b)})| + \cdots +|\zeta^{n-1}(u_{j-1}^{(b)})|\right)\right] \Phi_a(\zeta^{n-1}(u_j^{(b)}),\om)
$$
(if $j=1$, the expression reduces to $\exp(0)=1$ by definition).
Let
\be \label{def-Psi}
\vec{\Psi}^{(a)}_{n}(\om):=\left( \begin{array}{c} \Phi_a(\zeta^{n}(1),\om) \\ \vdots \\ \Phi_a(\zeta^{n}(m),\om) \end{array} \right)\ \ \ \mbox{and}\ \ \
\Pbi_n(\om) = [\vec{\Psi}_n^{(1)}(\om),\ldots, \vec{\Psi}_n^{(m)}(\om)],
\ee
where $\Pbi_n(\om)$ is the $m\times m$ matrix-function specified by its column vectors.
It follows that
\be \label{eq-matr1}
\Pbi_n(\om)=\M_{n-1}(\om) \Pbi_{n-1}(\om),\ \ n\ge 1,
\ee
where $\M_{n-1}(\om)$ is an $m\times m$ matrix-function, whose matrix elements  are trigonometric polynomials  given by
\be \label{matr}
(\M_{n-1}(\om))(b,c) = \sum_{j \le k_b:\ u_j^{(b)} = c} \exp\left[-2\pi i \om \left(|\zeta^{n-1}(u_1^{(b)})| + \cdots +|\zeta^{n-1}(u_{j-1}^{(b)})|\right)\right] 
\ee
(again the convention is that for $j=1$ we have $\exp(0)$).
Note that $\M_n(0) =\Sf^t= \Sf_\zeta^t$, the transpose of the substitution matrix, for all $n\in \N$. 

\medskip

\noindent {\bf Example.} {\em Let $\zeta$ be a substitution on $\{1,2\}$ given by $\zeta(1) = 1222,\ \zeta(2) =1$. Then $\Sf = \left( \begin{array}{cc} 1 & 1 \\ 3 & 0 \end{array} \right)$,
$$
\M_n(\om) = \left( \begin{array}{cc} 1 & e^{-2\pi i \om|\zeta^n(1)|} + e^{-2\pi i \om (|\zeta^n(1)|+|\zeta^n(2)|)} + e^{-2\pi i \om (|\zeta^n(1)|+2|\zeta^n(2)|) } \\ 1 & 0 \end{array} \right),\ \ \ n\ge 0.
$$
}

\medskip

Since $\vec{\Psi}_0^{(a)}(\om)= \bbe_a$ (the basis vector corresponding to $a\in \A$), it follows from (\ref{eq-matr1}) that
\be \label{eq-matr2}
\vec{\Psi}_{n}^{(a)}(\om) = \M_{n-1}(\om) \M_{n-2}(\om) \cdots \M_0(\om) \bbe_a,
\ee
hence
\be \label{def-Pbi}
\Pbi_n(\om)= \M_{n-1} (\om)\M_{n-2}(\om)\cdots \M_0(\om).
\ee 
Denote by $\Sigb_\zeta$ the $m\times m$ matrix of correlation measures for $\zeta$; that is, $\Sigb_\zeta(a,b) = \sig_{a,b}$. Below $\one$ denotes the vector $[1,\ldots,1]^t$ of all 1's and
$\langle\vec{x},\vec{y}\rangle$ stands for the scalar product in $\R^m$. We write $x_n\sim y_n$ when $\lim_{n\to \infty} x_n/y_n=1$.

\begin{lemma} \label{lem-Riesz}
Let $\theta$ be the Perron-Frobenius eigenvalue of the substitution matrix $\Sf$ and let $\vec{r},\vec{\ell}$ be respectively the (right) eigenvectors of $\Sf,\Sf^t$ corresponding to $\theta$, normalized by the condition $\langle\vec{r},\vec{\ell}\rangle =1$. Then
\be \label{eq-Sigb}
 \Sigb_\zeta = \frac{1}{\langle\vec{r},\one\rangle \langle\one,\vec{\ell}\rangle}\cdot\mbox{\rm weak*-}\lim_{n\to \infty} \theta^{-n} \ov{\Pbi_n^*(\om) \Pbi_n(\om)}\,d\om.
\ee
\end{lemma}

\begin{proof}
In view of (\ref{def-Psi}), the $(a,b)$ entry of $\ov{\Pbi_n^*(\om) \Pbi_n(\om)}$ is $\sum_{j=1}^{m} {\Phi_a(\zeta^n(j),\om)}\cdot\ov{\Phi_b(\zeta^n(j),\om)}$.
We have
$$
|\zeta^n(j)| = \langle\Sf^n\bbe_j, \one\rangle \sim \theta^n \langle \bbe_j, \vec{\ell}\rangle  \langle \vec{r}, \one\rangle
$$
by the Perron-Frobenius Theorem, hence
$$
\theta^{-n} \sum_{j=1}^{m} {\Phi_a(\zeta^n(j),\om)}\cdot\ov{\Phi_b(\zeta^n(j),\om)} \sim \sum_{j=1}^{m} \frac{\langle \bbe_j, \vec{\ell}\rangle \langle  \vec{r}, \one\rangle}{|\zeta^n(j)|} \ {\Phi_a(\zeta^n(j),\om)}\cdot\ov{\Phi_b(\zeta^n(j),\om)},
$$
and the desired claim follows from (\ref{corr2}), in view of the fact that  $\one = \sum_{j=1}^{m} \bbe_j$.
\end{proof}

\begin{remark}
Representing spectral measures as Riesz products for substitutions of {\em constant length} has been used for a long time, especially for the Thue-Morse substitution and its generalizations \cite{KS}. 
Queffelec \cite[Theorem 8.1]{Queff} proved a variant of Lemma~\ref{lem-Riesz} for substitutions of constant length. \end{remark}

\subsection{Other spectral measures and maximal spectral type}
Substitution $\zeta$ can also be extended to $\Ak^\Z$, so that $\zeta(x_0)$ starts from 0-th position. Consider  the sets $\zeta^k[a]$, $a\in \Ak$, $k\ge 0$. It is proved in \cite[5.6.3]{Queff} that
$$
\Pk_k = \{T_\zeta^i(\zeta^k[a]),\ a\in \Ak, \ 0 \le i < |\zeta^k(a)|\}
$$
is a Kakutani-Rokhlin partition for the substitution dynamical system, and $\{\Pk_k\}_{k\ge 0}$ generate the Borel $\sig$-algebra (this relies
on the fact that aperiodic primitive substitutions are {\em bilaterally recognizable}, see \cite{Queff}). It follows that the maximal spectral type of $(X_\zeta,T_\zeta,\mu)$ is equivalent to 
$$
\sum_{k\ge 0,\ a\in \Ak} 2^{-k} \sig_{\One_{\zeta^k[a]}}.
$$
The spectral measures $\sig_{\One_{\zeta^k[a]}}$ can be analyzed similarly to the correlation measures $\sig_a$. In fact, it is not hard to see that the matrix of measures
$$
\Sigb_\zeta^{(k)}:= [\sig_{\One_{\zeta^k[a]},\One_{\zeta^k[b]}}]_{a,b\in \Ak}
$$
can be expressed, analogously to (\ref{eq-Sigb}), as
$$
 \Sigb^{(k)}_\zeta = \frac{1}{\langle \vec{r},\one\rangle \langle \one,\vec{\ell}\rangle }\cdot\mbox{\rm weak*-}\lim_{n\to \infty} \theta^{-n} \ov{(\Pbi_n^{(k)})^*(\om) \Pbi^{(k)}_n(\om)} \,d\om,
$$
where
$$
\Pbi^{(k)}_n(\om) = \M_{n+k-1}(\om)\cdots \M_{k}(\om).
$$


\section{Estimating spectral measures via matrix Riesz products} \label{sec-Hof} 

In this section we study the local quantitative behavior of spectral measures; more precisely, upper bounds for the measures of small balls.
Using matrix products from Section 2, this problem is essentially reduced to some questions of Diophantine approximation.

It should be pointed out that matrix products, analogous to those that we consider, have been studied in the physics literature, in particular, in the papers by Aubry-Godr\`eche-Luck \cite{Aubry} and G\"ahler-Klitzing \cite{GK}. These papers study the diffraction spectrum of structures associated to substitutions, mostly in the continuous setting, as our self-similar suspension flows in Section 5, and their higher-dimensional generalizations (self-similar tilings). As was shown by S. Dworkin \cite{Dworkin} (see also \cite{Hof3}), the diffraction spectrum is a ``part'' of the dynamical spectrum that we consider. 
The papers \cite{Aubry,GK} analyzed the growth of expressions like our $|\Phi_a(\zeta^n(b),\om)|$ (called ``structure factor''), with the help of finite matrix products, and used them to make conclusions (sometimes heuristically) about the diffraction spectrum.
The paper
\cite{GK} focused on the discrete part of the  spectrum and showed that the ``maximal growth'' of $|\Phi_a(\zeta^n(b),\om)|$ occurs in the Pisot case on a dense set of spectral parameters, which, in dynamical terms means eigenvalues. This can be compared to a theorem  from \cite{SolTil} asserting
that the self-similar substitution tiling system is weakly mixing if and only if the expansion factor is non-Pisot. The paper \cite{Aubry} examined in detail some non-Pisot examples and obtained ``intermediate growth'' of the structure factor for a dense set of spectral parameters, from which singular-continuous diffraction spectrum was deduced heuristically. A rigorous argument in this direction was made by A. Hof \cite{Hof}; 
Lemmas \ref{lem-easy1} and  Lemma \ref{lem-var} below are based on it.
See also the papers \cite{GL,ChSa,KIR,Zaks} for related work in the physics literature. 


The next lemma is ``extracted'' from  \cite{Hof}. First we need to introduce some notation.
Let $(X,T,\mu)$ be a measure-preserving system. For $f\in L^2(X,\mu)$ let
\be \label{def-G}
G_N(f,\om) = N^{-1} \left\|\sum_{n=0}^{N-1} e^{-2\pi i n \om} f\circ T^n\right\|_{L^2(X)}^2.
\ee
Note that 
$$
\sig_f =  \mbox{\rm weak*-}\lim_{N\to \infty} G_N(f,\om)\, d\om
$$
by Lemma~\ref{lem-spec1}. 
For $x\in X$ and $f\in L^2(X,\mu)$ let
\be \label{def-SN}
S_N^x(f,\om) = \sum_{n=0}^{N-1}e^{-2\pi i n\om} f(T^n x),
\ee
so that
\be \label{eq-GN}
G_N(f,\om) = N^{-1} \int_X |S_N^x(f,\om)|^2\,d\mu(x).
\ee

\begin{lemma} \label{lem-easy1}
For all $\om\in[0,1)$ and $r\in (0,\half]$  we have
\be \label{eq-estim1}
\sig_f(B(\om,r))\le \frac{\pi^2}{4N} G_N(f,\om)\ \ \ \mbox{for}\ \ N = \lfloor (2r)^{-1}\rfloor.
\ee
\end{lemma}

\begin{proof}
As in the proof of Theorem 2.1 in \cite{Hof}, we have
$$
G_N(f,\om) = \int_{\T} K_{N-1} (\om-x)\,d\sig_f(x),\ \ \mbox{where}\ \ K_{N-1}(y) =
\frac{1}{N} \Bigl( \frac{\sin(N\pi y)}{\sin(\pi y)}\Bigl)^2
$$
is the Fej\'er kernel. Given $r\in (0,\half]$, let $N = \lfloor (2r)^{-1} \rfloor$. Then for
$|\om-x|\le r$ we have $N\pi |\om-x| \le \frac{\pi}{2}$, hence
$$
|\sin(N\pi (\om-x))| \ge (2/\pi) N\pi |\om-x|\ \ \mbox{and}\ \ |\sin(\pi(\om-x))| \le \pi|\om-x|.
$$
Therefore, $K_{N-1}(\om-x) \ge 4N/\pi^2$, and we obtain
$$
\sig_f(B(\om,r)) \le \int_{\om-r}^{\om+r}\frac{\pi^2}{4N}\,K_{N-1} (\om-x)\,d\sig_f(x) \le
\frac{\pi^2}{4N}\,G_N(f,\om),
$$ 
as desired.
\end{proof}



\begin{lemma} \label{lem-var}
 Let $\Om(r)$ be a continuous increasing function on $[0,1)$, such that $\Om(0)=0$, and suppose that for some fixed $\om \in [0,1)$ there exists $N_0\ge 1$ such that
\be \label{eq-Hof1}
G_N(f,\om) \le C N \Om(1/N)\ \ \mbox{for}\ N\ge N_0.
\ee
Then
\be \label{eq-Hof2}
\sig_f(B(\om,r)) \le \frac{\pi^2C}{4}\, \Om(3r) \ \ \mbox{for all}\ r \le (2N_0)^{-1}.
\ee
In particular, (\ref{eq-Hof2}) holds provided 
\be \label{eq-SN}
\sup_{x\in X} |S^x_N(f,\om)| \le N \sqrt{C\Om(1/N)}\ \ \mbox{for}\ N\ge N_0.
\ee
 \end{lemma}

\begin{proof}
It follows from (\ref{eq-GN}) that (\ref{eq-SN}) implies (\ref{eq-Hof1}), so we only need to show that (\ref{eq-Hof1}) implies (\ref{eq-Hof2}).
Given $r\le (2N_0)^{-1}\le \half$, take $N = \lfloor 1/2r \rfloor$. Then by Lemma~\ref{lem-easy1},
$$
\sig_f(B(\om,r)) \le \frac{\pi^2}{4N}\,G_N(f,\om)
\le \frac{\pi^2C}{4}\,\Om(1/N) \le \frac{\pi^2C}{4}\,\Om(3r), 
$$
since  $N^{-1}\le 3r$, and the proof is complete.
\end{proof}

\subsection{Estimating spectral measures for substitutions.}
Next we restrict ourselves to the substitution dynamical system
 $(X_\zeta,T_\zeta,\mu)$ and suppose that $f:\,X_\zeta\to \C$ depends only on the first symbol $x_0$, that is, $f=\sum_{a\in \Ak}d_a \One_{[\a]}$. 
 Then (\ref{def-Phi}) implies that
 \be \label{uga1}
 S_N^x(f,\om) = \sum_{a\in \Ak} d_a {\Phi_\a(x[0,N-1],\om)},\ \ \ \mbox{where}\ \ x[0,N-1]=x_0\ldots x_{N-1}.
 \ee
We will first show that, under some mild assumptions, $|S_N^x(f,\om)|$ may be estimated from above, uniformly in $x$, once
we have estimates for $|\Phi_a(\zeta^n(b),\om)|$. This will yield
upper bounds for $G_N(\One_{[a]},\om)$ via (\ref{eq-GN}) and then estimates of $\sig_f(B(\om,r))$ via Lemma~\ref{lem-easy1}. Denote
$$
\Phi_f(v,\om) = \sum_{a\in \Ak} d_a \Phi_a(v,\om)
$$
for $f=\sum_{a\in \Ak}d_a \One_{[\a]}$ and a word $v\in \Ak^+$.
 
\begin{prop} \label{prop-Step1}
Let $\zeta$ be a primitive substitution on $\Ak$, and let $\theta$ be the Perron-Frobenius eigenvalue of the substitution matrix $\Sf = \Sf_\zeta$.
Take a function $f=\sum_{a\in \Ak}d_a \One_{[\a]}$ and a number $\om \in [0,1)$, and suppose that there exists a sequence 
$\{F_\om(n)\}_{n\ge 0}$ satisfying
\be \label{estim3}
\frac{F_\om(n)}{\theta'} \le F_\om(n+1) \le F_\om(n),\ \ \ n\ge 0,
\ee
with $1< \theta'< \theta$, such that
\be \label{estim2}
|\Phi_f(\zeta^n(b),\om)|\le |\zeta^n(b)|F_\om(n),\ \ \ b\in \Ak,\ \ n\ge 0.
\ee
 Then 
\be \label{estim4}
|\Phi_f(x[0,N-1],\om)| \le \frac{C_1}{\theta-\theta'}\cdot N \cdot F_\om(\lfloor\log_\theta N - C_2\rfloor),
\ee
where $C_1,C_2>0$ depend only on $\zeta$.
\end{prop}

\begin{proof}
We shall need the following well-known prefix-suffix decomposition.

\begin{lemma} {\em (see Dumont and Thomas \cite[Th.\,1.5]{DT}, Rauzy \cite{Rauzy})} \label{lem-accord} Let $x\in X_\zeta$ and $N\geq 1$. Then
 \begin{equation} \label{eq-accord}
x[0,N-1]=u_0 \zeta(u_1)\zeta^2(u_2)\ldots \zeta^n(u_n)\zeta^n(v_n) \zeta^{n-1}(v_{n-1})\ldots \zeta(v_1)v_0,
\end{equation}
where $n\ge 0$ and $u_i,v_i,\ i=0,\ldots,n$, are respectively proper prefixes and suffixes  of the words $\zeta(b)$, $b\in \mathcal A$. The words $u_i, v_i$ may be empty, except that  at least one of $u_n, v_n$ is
nonempty.
\end{lemma}

Now we prove the proposition. By (\ref{eq-accord}) and (\ref{eq-Phi}), 
\begin{eqnarray}
|\Phi_f(x[0,N-1],\om)| & \le & \sum_{j=0}^n \bigl( |\Phi_f(\zeta^j(u_j),\om)| + |\Phi_f(\zeta^j(v_j),\om)|
\bigr) \nonumber \\ \label{eq-accord2}
& \le & 2L\sum_{j=0}^n \max_{b\in \A} |\Phi_f(\zeta^j(b),\om)|,
\end{eqnarray}
where $L=\max_{b\in \A} |\zeta(b)|$, using that $u_j,v_j$ are subwords of substituted symbols.
By the Perron-Frobenius Theorem, there exist $c,c'>0$ depending only on the substitution, such that 
\be \label{eq-PF1}
c\theta^j \le |\zeta^j(b)|\le c'\theta^j,\ \ \ j\ge 0,\ b\in \Ak.
\ee
Thus, we obtain from (\ref{eq-accord2}), (\ref{estim2}), and (\ref{estim3}): 
\begin{eqnarray*}
|\Phi_f(x[0,N-1],\om)| & \le & 2L \sum_{j=0}^n \max_{b\in \A} |\zeta^j(b)|  F_\om(j) \\
& \le & 2Lc' \sum_{j=0}^n \theta^j (\theta')^{n-j} F_\om(n) 
 < \frac{2Lc'\theta}{\theta-\theta'}\cdot \theta^n F_\om(n).
\end{eqnarray*}
It follows from (\ref{eq-accord}) and (\ref{eq-PF1}) that 
$$
c\theta^n\le \min_{b\in \A} |\zeta^n(b)| \le N \le 2\max_{b\in \A}|\zeta^{n+1}(b)|\le 2c'\theta^{n+1},
$$
whence 
$$
|\Phi_f(x[0,N-1],\om)|\le \frac{2Lc'\theta}{c(\theta-\theta')} \cdot N \cdot F_\om(\lfloor \log_\theta N - \log_\theta(2c')-1\rfloor),
$$
as desired. We used that $F_\om(n)$ is non-increasing in the last step.
This concludes the proof of the proposition.
\end{proof}

 In order to state the next result, we
 need some terminology.
A word $v$ is called a {\em return word} for the substitution $\zeta$ if $v$ starts with some letter $c$, doesn't contain other $c$'s, and $vc$ appears as a word in (any of) the sequences in $X_\zeta$ (by the minimality of the substitution dynamical system, the set of words appearing in $x\in X_\zeta$ is the same for all $x$). That is, the return word separates two successive occurrences of a given letter
(with the letter included at the beginning). 

Fix a return word $v$ starting with $c$.
We can replace $\zeta$ by $\zeta^\ell$ without loss of generality since this does not change the space $X_\zeta$ and hence the substitution system. Since $\zeta$ is primitive, we can thus assume
without loss of generality that $vc$ occurs as a subword in every $\zeta(b),\ b\in \A$. Let 
$$\|x\|:=\dist(x,\Z)\ \ \ \mbox{for}\ x\in \R.
$$ 
This is standard notation in the theory of Diophantine approximation; when we use $\|\cdot\|$ for a norm, this is always indicated by a subscript, as in (\ref{def-G}).

The following proposition is a key result: it shows (i) that it suffices to estimate (\ref{uga1}) for $x$ starting with $\zeta^n(b)$ for some $b\in \A$, and this in turn reduces to estimating matrix Riesz products; (ii) how to 
estimate matrix Riesz products in terms of ordinary products. 

\begin{prop}  \label{prop-Dioph}
Let $\zeta$ be a primitive aperiodic substitution on $\A$ and $v$ a return word starting with $c\in\A$ such that $vc$ occurs as a subword in  $\zeta(b)$ for every $b\in \A$. 
Let $a\in \A$, $\om\in \T$, and $S^x_N(\One_{[a]},\om)$ be defined by (\ref{def-SN}) with $T=T_\zeta$. Then there exist $c_1\in (0,1)$ and $C',C'',C_2>0$,  depending only on the substitution $\zeta$, such that 

{\bf (i)} for all $a,b\in \A$, $n\in \Nat$, and $\om\in [0,1)$,
\be \label{eq-new1}
|\Phi_a(\zeta^n(b),\om)|\le C' |\zeta^n(b)|\cdot \prod_{k=0}^{n-1} \Bigl(1 - c_1\bigl\|\om\,|\zeta^k(v)|\bigr\|^2\Bigr);
\ee

{\bf (ii)} for all $N\in \Nat,\ \om \in [0,1)$, and $a\in \A$,
\be \label{eq-matrest0}
|S^x_N(\One_{[a]},\om)| \le C'' N\prod_{k=0}^{\lfloor \log_\theta N- C_2\rfloor} (1 - c_1\bigl\|\om |\zeta^k(v)|\bigr\|^2)\ \ \ \mbox{for all}\ \ x\in X_\zeta.
\ee
\end{prop}

\begin{remark} \label{rem-Dioph}
This result should be compared with the criterion for $e^{2\pi i \om}$ to be an eigenvalue of the substitution dynamical system \cite{Host,FMN}, which, under some  technical assumptions, says that
\be \label{cond-eigen}
e^{2\pi i \om}\ \ \mbox{is an eigenvalue}\ \Longleftrightarrow\ \sum_{k=0}^\infty \big\|\om |\zeta^k(v)|\bigr\|^2<\infty
\ \ \mbox{for every return word $v$}.
\ee
Note that in the last proposition we do not exclude the case when $T_\zeta$ has nontrivial eigenvalues.
\end{remark}

Proposition~\ref{prop-Dioph} and Lemma~\ref{lem-var} imply that on certain subsets of $\T$ defined in terms of Diophantine properties, the spectral measures satisfy a 
H\"older condition with a uniform exponent.

\begin{corollary} \label{cor-Dioph} Suppose that the assumptions of Proposition~\ref{prop-Dioph} are satisfied.
Fix small $\delta>0, \ \eps>0,$ and consider
 \be \label{def-Om}
 \Gamma_{\delta,\eps}(\zeta):= \Bigl\{\om\in \T:\ \liminf_{n\to \infty} \frac{\#\{k\le n:\ \bigl\|\om |\zeta^k(v)| \bigr\| \ge\delta\}}{n} > \eps\Bigr\}.
 \ee
Then for all $a,b\in \A$ and $\om \in \Gamma_{\delta,\eps}(\zeta)$ there exists $C>0$ such that 
\be \label{eq-Dioph2}
\sig_a(B(\om,r)) \le C r^\beta,\ \ \mbox{where}\ \beta = \beta(\zeta,\delta,\eps).
\ee
\end{corollary}

\begin{proof}
It follows from Proposition~\ref{prop-Dioph} that (\ref{eq-Hof1}) holds with $\Om(r) = r^{-2\eps \log_\theta(1-c_1\delta^2)}$ for $f = \One_{[a]}$, so Lemma~\ref{lem-var} yields (\ref{eq-Dioph2}),
since $c_1$ depends only on the substitution $\zeta$. 
\end{proof}

\begin{remark} 
1. One can show that  $$\dim_H\Bigl(\T\setminus \bigcup_{\delta>0,\eps>0} \Gamma_{\delta,\eps}(\zeta)\Bigr) = 0.$$
For instance, in the special case when $\zeta$ has constant length $q\ge 2$, the set 
$\T\setminus \bigcup_{\eps>0} \Gamma_{q^{-\ell},\eps}(\zeta)$ consists of those $\om$ for which the standard base-$q^\ell$ expansion of
$\om|v|$ has all digits, except  $0$ and $q^{\ell}-1$, with zero frequency, whence a standard computation yields
$$
\dim_H\Bigl(\T\setminus \bigcup_{\eps>0} \Gamma_{q^{-\ell},\eps}(\zeta)\Bigr) = \frac{\log 2}{\log(q^\ell)} \to 0,\ \ \mbox{as}\ \ell\to \infty.
$$

2. It follows from \cite{Weyl} that for Lebesgue-a.e.\ 
$\om$ the sequence $\{\om |\zeta^k(v)|\}_{k\ge 0}$ is uniformly distributed modulo 1, which yields an explicit
H\"older exponent for spectral measures $\sig_a$ almost everywhere. This is, however, not very interesting, since any finite positive measure on $[0,1]$ has an a.e.\ differentiable cumulative distribution function, hence that measure satisfies a Lipschitz condition almost everywhere.

\end{remark}

\begin{proof}[Proof of Proposition~\ref{prop-Dioph}] 
{\bf Step 1: reduction (\ref{eq-new1}) $\Rightarrow$ (\ref{eq-matrest0}).} This is immediate from Proposition~\ref{prop-Step1}: take $f=\One_{[a]}$,
$$
F_\om(n) := C'\prod_{k=0}^{n-1} \Bigl(1 - c_1\bigl\|\om\,|\zeta^k(v)|\bigr\|^2\Bigr),
$$
and note that this sequence satisfies $(1-c_1) F_n(\om) \le F_{n+1}(\om) \le F_n(\om)$. Now,
assuming
$$
c_1 \le \frac{\theta-1}{\theta+1},
$$
the condition (\ref{estim3}) holds with $\theta' = (1+\theta)/2$, and Proposition~\ref{prop-Step1} applies.

\medskip

{\bf Step 2: proof of (\ref{eq-new1}).} We fix $\om$ and omit it from  notation, so that $\M_n=\M_n(\om)$. We are going to use (\ref{def-Psi}) and (\ref{eq-matr2}). By assumption,  for any $b\in \A$, we can write 
\be \label{eq-ret}
\zeta(b) =  p^{(b)} vc q^{(b)},
\ee
where $p^{(b)}$ and $q^{(b)}$ are words, possibly empty, and $v$ starts with $c$. We are going to estimate the absolute value of the trigonometric polynomial $\M_n(b,c)$ given
by (\ref{matr}). Note that $\M_n(b,c)$ is  a trigonometric polynomial with $\Sf^t(b,c)$ exponential terms and all coefficients equal to one. 
By (\ref{matr}) and (\ref{eq-ret}),  the expression for $\M_n(b,c)$ includes the terms $e^{-2\pi i \om |\zeta^n(p^{(b)})|}$ and $e^{-2\pi i \om |\zeta^n(p^{(b)}v)|}$, hence 
$$
|\M_n(b,c)| \le \Sf^t(b,c) - 2 + |1 + e^{-2\pi i \om |\zeta^n(v)|} |.
$$
From the inequality
\be \label{lem-elem1}
|1 + e^{2\pi i \tau} | \le 2 - \half \|\tau\|^2,\ \ \tau\in \R,
\ee
we have
\be \label{eq-new11}
|\M_n(b,c)|\le \Sf^t(b,c) - \half \bigl\|\om |\zeta^n(v)| \bigr\|^2.
\ee
We will use the following notation: 
\begin{itemize} \item for  vectors $\vec{x},\vec{y}\in \R^m$, the inequality $\vec{x}\le \vec{y}$ means componentwise inequality;
\item
the operation of taking absolute values of all entries for a vector $\vec{x}$  and a matrix $A$ will be denoted $\vec{x}^{\mathbf |\cdot|}$ and $A^{|\cdot|}$.
\end{itemize}
%
%
It is clear that for any, generally speaking, rectangular matrices $A,B$ such that the product $AB$ is well-defined, we have
\be \label{eq-useful}
(AB)^{|\cdot|}\le A^{|\cdot|} B^{|\cdot|}.
\ee 
For an arbitrary $\vec{x} = [x_1,\ldots,x_{m}]^t >\vec{0}$ and $k\ge 0$, using (\ref{eq-new11}) we can estimate
\begin{eqnarray}
(\M_k^{|\cdot|}\vec{x})_b & = & \sum_{j=1}^{m}|\M_k(b,j)| x_j \nonumber \\
                                & \le & \sum_{j=1}^{m}
                             \Sf^t(b,j) x_j - \half \bigl\|\om |\zeta^k(v)| \bigr\|^2x_c\nonumber \\
                                 & \le & (1-c_3(\vec{x}) \bigl\|\om |\zeta^k(v)| \bigr\|^2) \cdot \sum_{j=1}^{m} \Sf^t(b,j) x_j \nonumber \\
                                  & =   & (1-c_3(\vec{x}) \bigl\|\om |\zeta^k(v)| \bigr\|^2) \cdot (\Sf^t \vec{x})_b, \label{eq-new12}
\end{eqnarray}
where
$$
c_3(\vec{x}) = \frac{x_c}{2m\max_j \Sf^t(b,j)\cdot \max_j x_j}\,.
$$
Observe that (\ref{eq-useful}) implies
$$
(\vec{\Psi}_n^{(a)}(\om))^{|\cdot|} = (\M_{n-1}\cdots \M_0 \vec{e}_a)^{|\cdot|}  \le  \M_{n-1}^{|\cdot|}\cdots \M_0^{|\cdot|} \vec{e}_a \le \M_{n-1}^{|\cdot|}\cdots \M_0^{|\cdot|} \one.
$$
Thus, using (\ref{eq-new12}) inductively we obtain
$$
(\vec{\Psi}_n^{(a)}(\om))^{|\cdot|} \le \prod_{k=0}^{n-1} \left(1- c_3((\Sf^t)^k \one)\bigl\|\om |\zeta^k(v)| \bigr\|^2\right) \cdot(\Sf^t)^n \one.
$$
Note that
$$
c_1 := \inf_{k\ge 0} c_3((\Sf^t)^k \one)>0
$$
by the Perron-Frobenius Theorem. Further, $\langle \vec{e}_j,(\Sf^t)^n \one\rangle = |\zeta^n(j)|$ for all $j$, hence
$(\Sf^t)^n \one\le C' |\zeta^n(b)|\one$, with $C'=c'/c$, by (\ref{eq-PF1}), and the desired inequality (\ref{eq-new1}) follows,  in view of (\ref{def-Psi}).
Proposition~\ref{prop-Dioph} is proved completely.
\end{proof}

\subsection{Spectral measure at zero}
As we saw in Lemmas~\ref{lem-easy1} and \ref{lem-var}, the decay of the spectral measure $\sig_f$ at zero is controlled, in some sense, by the  growth of Birkhoff sums
$
S_N^x(f,0) = \sum_{n=0}^{N-1} f(T^n x).
$
Of course, we should take $f$ orthogonal to constants, in order to avoid the trivial point mass at zero. In the case of a substitution dynamical
system $(X_\zeta,T_\zeta,\mu)$, it is natural to consider functions $f: X_\zeta\to \C$ depending only on the first coordinate $x_0$, that is,
$f = \sum_{a\in \Ak} d_a \One_{[a]}$. So we assume $\int f\,d\mu=\sum_{a\in \Ak} d_a\mu([a])=0$. (It is well-known that $(\mu([a]))_{a\in \Ak}$ is the normalized
Perron-Frobenius eigenvector of the substitution matrix $\Sf_\zeta$, see \cite{Queff}.) We have
$$
S_N^x(f,0) = \sum_{n=0}^{N-1} d_{x_n} .
$$
The asymptotics and sharp estimates of $S_N^x(f,0)$ were investigated by J.-M. Dumont and A. Thomas \cite{DT}, B. Adamczewski \cite{Adam},
and others. 
We do not give a complete overview here, but state two results that are relevant for us.

\begin{theorem} \label{th-DT} {\em (Dumont and Thomas \cite[Th.\,2.6]{DT})} Let $\zeta$ be a primitive substitution, and suppose that it has a fixed point $\wtil{x}\in \Ak^\N$, that is, $\zeta(\wtil{x})=\wtil{x}$.
Assume that the   second eigenvalue of the substitution matrix satisfies $\theta_2>1$, with multiplicity $\alpha+1$, and all other eigenvalues, except the largest $\theta$,
are strictly less that $\theta_2$ in absolute value. Suppose that $f = \sum_{a\in \Ak} d_a \One_{[a]}$, with $\int f\,d\mu=0$. Then there exists a continuous function $G$ on $(0,\infty)$, such that $G(\theta t) = G(t)$ for
all $t>0$ and 
$$
S_N^{\wtil{x}}(f,0) = (\log_\theta N)^\alpha N^\beta G(N) +o((\log N)^\alpha N^\beta),\ \ \ N\to \infty,\ \ \ \mbox{where}\ \ \beta = \log_\theta \theta_2.
$$
\end{theorem}

From this theorem, combined with Proposition~\ref{prop-Step1}, it  follows that 
$$
S_N^x(f,0) = O((\log_\theta N)^\alpha N^\beta),\ \ \ N\to\infty,
$$
uniformly in $x\in X_\zeta$. 
In a much more general case, Adamczewski \cite{Adam} obtained sharp estimates for $S_N^x(f,0)$. The following is a Corollary of \cite[Th.\,1]{Adam}
(we do not give a complete statement here).

\begin{theorem} \label{th-Adam} {\em (Adamczewski \cite{Adam})} Let $\zeta$ be a primitive substitution, with the second eigenvalue of the substitution matrix equal to $\theta_2$,
and the total number of eigenvalues equal to $\theta_2$ in absolute value is $\alpha+1$. 
Suppose that $f = \sum_{a\in \Ak} d_a \One_{[a]}$, with $\int f\,d\mu=0$. Then the following holds for a fixed point of the substitution 
$\wtil{x}$:

{\bf (i)} if $|\theta_2|>1$, then $S_N^{\wtil{x}}(f,0) = O((\log_\theta N)^\alpha N^\beta)$, with $\beta = \log_\theta \theta_2$;

{\bf (ii)} if $|\theta_2|=1$, then $S_N^{\wtil{x}}(f,0) = O((\log_\theta N)^{\alpha+1})$.
\end{theorem}

Now the following is immediate, in view of Lemma~\ref{lem-var} and Proposition~\ref{prop-Step1}.

\begin{corollary}
Under the conditions of Theorem~\ref{th-Adam} we have

{\bf (i)} if $|\theta_2|>1$, then $\sig_f(B(0,r)) = O((\log(1/r)^{2\alpha} r^{2-2\beta}),\ r\to 0$;

{\bf (ii)} if $|\theta_2|=1$, then $\sig_f(B(0,r)) = O((\log(1/r)^{2(\alpha+1)} r^{2}),\ r\to 0$.
\end{corollary}


\section{H\"older continuity of  spectral measures for  almost every substitution suspension flow}

Let $\zeta$ be a primitive substitution on $\Ak=\{1,\ldots,m\}$, and let $(X_\zeta,T_\zeta)$ be the corresponding uniquely ergodic $\Z$-action. For a strictly positive vector $\vec{s} = (s_1,\ldots,s_m)$ we consider the suspension flow over $T_\zeta$, with the piecewise-constant roof function, equal to $s_j$ on the cylinder set $[j]$. The resulting space will be denoted by $\Xxi^{\vec{s}}$ and the flow
by $(\Xxi^{\vec{s}},h_t)$. This flow can also be viewed as the translation action on a tiling space, with interval prototiles of length $s_j$. Such dynamical systems have been studied e.g.\ in \cite{BeRa,CSa}.
Denote by $\wtil{\mu}$ the unique invariant probability measure for the suspension flow $(\Xxi^{\vec{s}},h_t)$. 
We have, by definition,
$$
\Xxi^{\vec{s}} = \bigcup_{a\in \A}  \Xx_a,\ \ \ \mbox{where}\ \ \Xx_a= \{(x,t):\ x\in X_\zeta,\ x_0=a,\ 0 \le t \le s_a\}. 
$$
and this union is disjoint in measure. 

Suppose that the characteristic polynomial $P_{\Sf}(t)$ of the substitution matrix $\Sf$ is irreducible.
The goal of this section is to obtain H\"older estimates for spectral measures of substitution suspension flows under the assumption that the
second eigenvalue of the substitution matrix $\Sf$ satisfies $|\theta_2|>1$.
This is a natural assumption: if $|\theta_2|<1$ (the ``PV case''), then all suspension flows have dense point spectrum (see \cite{CSa}) and
hence H\"older estimates cannot hold. On the other hand, if $|\theta_2|\ge 1$, then (assuming $P_{\Sf}(t)$ is irreducible), almost every suspension flow
is weakly mixing by \cite{CSa} (see, in particular, \cite[Theorem 2.7]{CSa}). ``Almost every'' refers to the choice of the vector $\vec{s}$, with respect
to the Lebesgue measure. We are able  to show that if $|\theta_2|>1$, then 
H\"older estimates hold for a.e.\ $\vec{s}$ (actually, a little bit stronger, in terms of the dimension of the exceptional set, see Theorem~\ref{th-holder2} below). Thus the spectrum is quantitatively ``separated'' from being discrete. In Section 7 we express this more precisely, in 
terms of dimensions of spectral measures. We do not know what happens in the borderline ``Salem case'' (see Definition~\ref{def-Pisot}), but we
expect that H\"older estimates no longer hold.

Let $f\in L^2(\Xxi^{\vec{s}},\wtil{\mu})$. By the Spectral Theorem for measure-preserving flows, there is a finite positive Borel measure $\sig_f$ on $\R$ such that
$$
\int_{-\infty}^\infty e^{2 \pi i\om t}\,d\sig_f(\om) = \langle f\circ h_t, f\rangle\ \ \ \mbox{for}\ t\in \R.
$$
We will focus on the spectral measures of  characteristic functions of $\Xx_a$:
$$
\sig_a:=\sig_f\ \ \mbox{for}\ f = \One_{\Xx_a},\ a\in \Ak.
$$

For a word $v$ in the alphabet $\Ak$ denote by $\vec{\ell}(v)\in \Z^m$ its ``population vector'' whose $j$-th entry is the number of $j$'s in $v$, for $j\le m$. We will  need the
``tiling length'' of $v$ defined by 
\be \label{tilength}
|v|_{\vec{s}}:= \langle\vec{\ell}(v), \vec{s}\rangle.
\ee
It is not hard to show, similarly to the case of substitution $\Z$-actions, as discussed in
Section 2.4, that the maximal spectral type of the flow $(\Xxi^{\vec{s}},h_t)$ is equivalent to 
\be \label{max_type}
\sum_{k\ge 0,\ a\in \Ak} 2^{-k} \sig_{\zeta^k[a]},
\ee
where $\sig_{\zeta^k[a]}$ is the spectral measure of the characteristic function of the set
$$
\Xx_{\zeta^k[a]}:= \{h_t(x,0):\ x\in \zeta^k[a],\ 0 \le t \le |\zeta^k(a)|_{\vec{s}}\},
$$
and these spectral measures can be analyzed similarly to $\sig_a$.

Suspension flows for $\vec{s}$, which differ by a constant multiple, are related by a ``time-scale change''.
It follows that if $\vec{s'}= c\vec{s}$ for $c>0$, then
\be \label{norma}
\sig_a^{\vec{s'}}(E) = \sig_a^{\vec{s}}(c^{-1}E)\ \ \ \mbox{for Borel}\ E,
\ee
where the superscript indicates which suspension flow is considered. Thus
it makes sense
to normalize suspension flows, for instance, assuming that $\vec{s} \in \Delta^{m-1}:= \{\vec{y}\in \R^m_+:\ \sum_{j=1}^m y_j =1\}$. 

\begin{theorem} \label{th-holder1} Let $\zeta$ be a primitive aperiodic substitution on $\Ak=\{1,\ldots,m\}$, with substitution matrix $\Sf$.  Suppose that the characteristic polynomial $P_{\Sf}(t)$ is irreducible and the second eigenvalue  satisfies $|\theta_2|>1$. Then for Lebesgue-almost every suspension flow the spectral measures of the dynamical system $(\Xxi^{\vec{s}},h_t)$ are H\"older continuous at points away from zero. 

More precisely, there exists a constant $\gam>0$, depending only on the substitution $\zeta$, such that for Lebesgue-almost every $\vec{s}\in  \Delta^{m-1}$ and $B>1$, there exist $r_0=r_0(\vec{s},B)>0$ and 
$C=C(\vec{s},B)>0$ such that  
\be \label{holder1}
\sig_a ([\om-r,\om+r]) \le C r^\gam,\ \ \mbox{for all}\  |\om| \in [\Bu^{-1},\Bu],\  0 < r \le r_0,\ \mbox{and}\ a\in \Ak.
\ee
\end{theorem}

In fact, we get an estimate of the dimension of the exceptional set of suspension flows.

\begin{theorem} \label{th-holder2}
Let $\zeta$ be a primitive aperiodic substitution on $\Ak=\{1,\ldots,m\}$, with substitution matrix $\Sf$. Suppose that $P_{\Sf}(t)$ is irreducible, and there are exactly $q$ eigenvalues of absolute value
$\le 1$, for some $0 \le q< m-1$.  Then for every $\eta>0$ there exists $\gam = \gam(\eta)>0$ and an exceptional set $\Ek_\eta \subset \Delta^{m-1}$ of Hausdorff dimension at most $q+\eta$ such that for every
$\vec{s} \in \Delta^{m-1} \setminus \Ek_\eta$ and any $B>1$ there exist 
$r_0=r_0(\vec{s},B)>0$ and 
$C=C(\vec{s},B)>0$ 
for which
we have the estimate (\ref{holder1}).
\end{theorem}

Note that Theorem~\ref{th-holder1} is an immediate corollary of Theorem~\ref{th-holder2}: just choose $\eta<1$ and note that $q+\eta< m-1 = \dim(\Delta^{m-1})$. 

For the proof of Theorem~\ref{th-holder1} (and also for Theorem~\ref{th-main2} in the next section),
we need an analog of Lemma~\ref{lem-var} for flows ($\R$-actions), which has also been essentially worked out by Hof \cite{Hof}. Let $(Y,\mu,h_t)$ be a measure-preserving flow.
For $f \in L^2(Y,\mu)$, $R>0$,  $\om \in \R^d$, and $y\in Y$ let
$$
G_R(f,\om) = R^{-1} \left\|\int_0^R e^{-2\pi i \om t} f( h_t y) \,dt\right\|_{L^2(Y)}^2\  \ \mbox{and}\ \  S_R^y(f,\om) = \int_0^R e^{-2\pi i \om t} f(h_t y)\,dt,
$$
so that
\be \label{eq-GN1}
G_R(f,\om) = R^{-1} \int_Y |S_R^y(f,\om)|^2\,d\mu(y).
\ee

 \begin{lemma} \label{lem-var2}
 Let $\Om(r)$ be a continuous increasing function on $[0,1)$, such that $\Om(0)=0$, and suppose that for some fixed $\om \in \R$, $R_0\ge 1$, $C>0$, and $f\in L^2(Y,\mu)$ we have
\be \label{eq-Hof11}
G_R(f,\om) \le C R \,\Om(1/R)\ \ \mbox{for}\ R\ge R_0.
\ee
Then
\be \label{eq-Hof12}
\sig_f([\om-r,\om+r]) \le \frac{\pi^2 C}{4} \, \Om(2r) \ \ \mbox{for all}\ r \le (2R_0)^{-1}.
\ee
In particular, if
\be \label{eq-SR}
\sup_{y\in Y} |S^y_R(f,\om)| \le R \sqrt{C\Om(1/R)}\ \ \mbox{for}\ R\ge R_0,
\ee
then (\ref{eq-Hof12}) holds. 
 \end{lemma}

The proof, which  follows Hof \cite{Hof}, is deferred to the Appendix.

\medskip

Next, consider an analog of (\ref{def-Phi}): for $v=v_0\ldots v_{N-1}\in \Ak^+$ let
\be \label{def-Phi3}
\Phi_a^{\vec{s}}(v,\om) = \sum_{j=0}^{N-1} \delta_{v_j,a} \exp(-2\pi i \om |v_0\ldots v_j|_{\vec{s}}).
\ee
Then 
\be \label{SR1}
S_R^{(x,0)}(\One_{\Xx_a},\om) = \frac{1-e^{-2\pi i \om s_a}}{2\pi i \om} \cdot {\Phi_a^{\vec{s}}(x[0,N-1],\om)}\ \ \ \mbox{for}\ \ R = \left|x[0,N-1]\right|_{\vec{s}}.
\ee
Moreover,
\be \label{SR2}
\left|S_{R+t}^{(x,t)}(\One_{\Xx_a},\om)-S_R^{(x,0)}(\One_{\Xx_a},\om)\right| \le t,\ \ \mbox{for}\ 0 < t \le s_{x_0},
\ee
and
\be \label{SR3}
\left|S_R^{(x,t)}(\One_{\Xx_a},\om)-S_{R'}^{(x,t)}(\One_{\Xx_a},\om)\right| \le |R-R'|.
\ee
We have an  analog of Proposition~\ref{prop-Dioph} for suspension flows:

\begin{prop}  \label{prop-Dioph0}
Let $\zeta$ be a primitive substitution on $\A$ and $v$ a return word starting with $c\in\A$ such that $vc$ occurs as a subword in  $\zeta(b)$ for every $b\in \A$. Let $\vec{s} \in \Delta^{m-1}$.
Then there exist $c_1\in (0,1)$ and $C,C',C_2>0$, depending only on the substitution $\zeta$ and $\min_j s_j$, such that 

{\bf (i)} for all $a,b\in \A$, $n\in \Nat$, and $\om\in \R$,
\be \label{eq-new19}
|\Phi_a^{\vec{s}}(\zeta^n(b),\om)|\le C |\zeta^n(b)|_{\vec{s}} \cdot \prod_{k=0}^{n-1} \Bigl(1 - c_1\bigl\|\om\,|\zeta^k(v)|_{\vec{s}}\bigr\|^2\Bigr);
\ee

{\bf (ii)} for all $R>1,\ \om \in \R$, and $a\in \A$,
\be \label{eq-matrest10}
|S^{(x,t)}_R(\One_{\Xx_a},\om)| \le C' R\prod_{k=0}^{\lfloor \log_\theta R- C_2\rfloor} (1 - c_1\bigl\|\om |\zeta^k(v)|_{\vec{s}}\bigr\|^2)\ \ \ \mbox{for all}\ \ (x,t) \in \Xxi^{\vec{s}}.
\ee
\end{prop}

\begin{proof}[Proof]
The proof is similar to that of Proposition~\ref{prop-Dioph}. Just replace the usual length of words by their ``tiling length,'' defined in (\ref{tilength}).
The constants in the proof will depend on the ratio $$\frac{\max_j s_j}{\min_j s_j}\le (\min_j s_j)^{-1}.$$
The implication (i) $\Rightarrow$ (ii) is obtained as in Step 1 of the proof of Proposition~\ref{prop-Dioph}, using an obvious analog
of Proposition~\ref{prop-Step1}, as well as (\ref{SR1})-(\ref{SR3}). For the proof of (i) we use the equality, obtained repeating the arguments from Section 2.3:
$$
\Phi_a^{\vec{s}}(\zeta^n(b),\om) =\left\langle\M_{n-1}^{\vec{s}}(\om)\cdots \M_0^{\vec{s}}(\om)\vec{e}_a,\vec{e}_b\right\rangle,
$$
where
\be \label{matrices2}
(\M_{\ell}^{\vec{s}}(\om))(b,c) = \sum_{j \le k_b:\ u_j^{(b)} = c} \exp\left[-2\pi i \om |\zeta^{\ell}(u_1^{(b)}\ldots u_{j-1}^{(b)})|_{\vec{s}}\right],\ \ \ell\in \Nat,
\ee
by analogy with (\ref{matr}). The matrix product is then estimated as in the proof of Proposition~\ref{prop-Dioph}.
\end{proof}

\begin{proof}[Proof of Theorem~\ref{th-holder2}]
Recall that, passing to a power $\zeta^\ell$ if necessary, we can always obtain a return word $v$ as in the statement of Proposition~\ref{prop-Dioph0}, and the existence of such a word (for $\zeta$ itself)  will be the standing assumption until the end of the section.


Let $\theta_1=\theta, \theta_2,\ldots,\theta_m$ be the eigenvalues of the  substitution matrix $\Sf$, ordered by magnitude, and let $\vec{e^*_j}$ be the corresponding eigenvectors of its transpose $\Sf^t$ (real and complex). By the assumptions on the matrix $\Sf$, it is diagonalizable over $\C$ and 
$$
|\theta_{m-q}|>1,\ \ \ |\theta_{m-q+1}|\le 1
$$
(we do not exclude the possibility of $q=0$; in that case the second inequality is vacuous).
Let $\{\vec{e}_j\}_1^m$ be the dual basis, i.e.\ $\vec{e}_j$ is the eigenvector of $\Sf$ corresponding to $\theta_j$ and $\langle \vec{e}_i, \vec{e_j^*}\rangle = \delta_{ij}$. 
Then $\vec{s} = \sum_{j=1}^m \langle \vec{e}_j,\vec{s}\rangle \vec{e_j^*}$, hence
$$
|\zeta^n(v)|_{\vec{s}}=\langle \vec{\ell}(\zeta^n(v)), \vec{s}\rangle = \langle \Sf^n \vec{\ell}(v),\vec{s}\rangle = \sum_{j=1}^m \langle \vec{e}_j,\vec{s}\rangle \, \langle \vec{\ell}(v),\vec{e_j^*}\rangle \,\theta_j^n,\  \ n\ge 0.
$$
Let
\be \label{coord}
b_j = \langle \vec{e}_j,\vec{s}\rangle \, \langle \vec{\ell}(v),\vec{e_j^*}\rangle,\ j=1,\ldots,m,
\ee
so that
\be \label{coord2}
|\zeta^n(v)|_{\vec{s}}= \sum_{j=1}^m b_j \theta_j^n.
\ee
We always have $b_1>0$, since $\theta_1$ is the Perron-Frobenius eigenvalue, both eigenvectors $\vec{e}_1$ and $\vec{e_1^*}$ are strictly positive, $\vec{s}$ is strictly positive, and $\vec{\ell}(v)\ne \vec{0}$ is
non-negative. Further, since $\vec{\ell}(v)$ is an integer vector and the characteristic polynomial of $\Sf$ is irreducible, we have $\langle \vec{\ell}(v),\vec{e_j^*}\rangle \ne 0$ for all $j\le m$. Indeed, otherwise $\Sf$ would have a rational invariant subspace, spanned by $\Sf^n \vec{\ell}(v),\ n\ge 0$, of dimension less than $m$, contradicting the fact that its eigenvalues are algebraic integers of degree $m$. Note also that
$b_{j'}=\ov{b_j}$ for $\theta_{j'} = \ov{\theta_j}$.
Let
$$
\Hk^{m-1} = \{(a_1,\ldots,a_m)\in \C^{m}:\ a_1=1,\ a_{j'} = \ov{a_j}\ \mbox{for}\ \theta_{j'} = \ov{\theta_j}\},
$$
let $P_{m-q}$ be the projection from $\Hk^{m-1}$ to the subspace spanned by the first $m-q$ coordinates, and let $\Hk^{m-q-1} = P_{m-q}\Hk^{m-1}$. It is clear that $\Hk^{m-1}$ is a real affine-linear space of dimension $m-1$ and $\Hk^{m-q-1}$ is a real affine-linear space
of dimension $m-q-1$.  It is convenient to pass from $\Delta^{m-1}$ to a subset of $\Hk^{m-1}$ when parametrizing the suspension flows. To this end, consider the map $\Fk:\,\C^{m-1} \to \C^{m-1}$ given by
\be \label{def-Fk}
\Fk(\vec{s}) = \left(\frac{\langle \vec{e}_j,\vec{s}\rangle \, \langle \vec{\ell}(v),\vec{e_j^*}\rangle}{\langle \vec{e}_1,\vec{s}\rangle \, \langle \vec{\ell}(v),\vec{e_1^*}\rangle}\right)_{1\le j \le m}.
\ee
The map $\Fk$ is by definition  a change of basis transformation, followed by an invertible diagonal map. Consequently, the map $\Fk$ is linear and invertible.  
Let $\wtil{\Delta}^{m-1}:= \{\vec{y}\in \R^m:\ \sum_{j=1}^m y_j =1\}$ be the affine-linear space spanned by $\Delta^{m-1}$.
Note that $\Fk(\wtil{\Delta}^{m-1})\subset \Hk^{m-1}$, and since both $\wtil{\Delta}^{m-1}$ and $\Hk^{m-1}$ are affine-linear spaces of real dimension $m-1$, we have that $\Fk|_{\wtil{\Delta}^{m-1}}$ is a real affine-linear invertible map onto $\Hk^{m-1}$, which preserves Hausdorff dimension.

The following proposition contains the core of the proof of Theorem~\ref{th-holder2}. We will need the Vandermonde matrix
\be \label{Vand}
\Theta = \left( \begin{array}{ccc} 1 & \ldots & 1 \\ \vdots & \ddots & \vdots \\  \theta_1^{m-1} & \ldots & \theta_m^{m-1} \end{array} \right)
\ee
and its $\ell^\infty$ operator norm $\|\Theta\|_\infty$; note that $\Theta$ is invertible, since all $\theta_j$ are distinct. 

\begin{prop} \label{prop-EKvar} Let $\Bu>1$ and  $k\in \Nat$. Consider two constants, depending only on the substitution matrix (actually, on
$\Theta$), defined as follows:
\be \label{def-Lrho}
\rho:= \half (1+\theta_1\|\Theta\|_\infty\|\Theta^{-1}\|_\infty)^{-1}\ \ \ \mbox{and}\ \ \ L:=2+\theta_1\|\Theta\|_\infty\|\Theta^{-1}\|_\infty.
\ee
Let $E_k^N(B)$ be the set of
$(a_1,\ldots,a_{m-q})\in \Hk^{m-q-1}$  such that there exist $\om\in [\Bu^{-1},\Bu]$ and $a_{m-q+1},\ldots,a_m$, with $(a_1,\ldots,a_m)\in \Fk(\Delta^{m-1})$, for which
\be \label{eq2}
 \card \Bigl\{n\in [1,N]:\ \Bigl\|\om\sum_{j=1}^m a_j \theta_j^n\Bigr\|\ge \rho\Bigr\} < \frac{N}{k}\,.
\ee
Further, let $E_k(B):=\bigcap_{N_0=1}^\infty \bigcup_{N=N_0}^\infty E_k^N(B)$. 
Then 
\be \label{upper}
\dim_H (E_k(B)) \le  \frac{\log[2 L^{m+1} k]}{k\log|\theta_{m-q}|}\,.
\ee
\end{prop}

We first derive Theorem~\ref{th-holder2} from Proposition \ref{prop-EKvar}. 
Let $$ E_k = \bigcup_{B>1} E_k(B).$$
Choose $k\in \Nat$ in such a way that $\dim_H(E_k(B)) <\eta$ for all $B>0$, which is possible since the right-hand side of (\ref{upper}) tends
to zero as $k\to\infty$.
Then we also have
$$
\dim_H(E_k)< \eta.
$$
Let
$$
\Ek_\eta:= (\Fk|_{\wtil{\Delta}^{m-1}})^{-1} P_{m-q}^{-1}(E_k).
$$
Note that $P_{m-q}^{-1}(E_k)$ is the direct product of $E_k$ with a real $q$-dimensional linear space, hence
$\dim_H(\Ek_\eta)=\dim_H(E_k)+q<\eta+q$.  We want to show that $\Ek_\eta$ is the desired exceptional set in Theorem~\ref{th-holder2}. To this end, let $\vec{s}\in \Delta^{m-1}\setminus \Ek_\eta$ and  $B>1$. Consider the coefficients $b_j$ defined by (\ref{coord}), so that (\ref{coord2}) holds; then $$\Fk(\vec{s}) = (1, b_2/b_1,\ldots, b_m/b_1)=:(a_1,\ldots,a_{m}).$$
Observe that
$$
b_1 = \langle \vec{e}_1,\vec{s}\rangle \, \langle \vec{\ell}(v),\vec{e_1^*}\rangle\in [C_3^{-1},C_3],
$$
where $C_3>1$ depends only on $\zeta$ and $v$, since 
$$
\min_j(\vec{e}_1)_j\le \langle \vec{e}_1,\vec{s}\rangle\le \max_j(\vec{e}_1)_j\ \ \mbox{for all}\ \vec{s} \in \Delta^{m-1}.
$$
By assumption,
$$(a_1,\ldots,a_{m-q}) \not\in E_k(C_3B),$$ hence there exists $N_0=N_0(\vec{s},C_3B)\in \Nat$ such that $$(a_1,\ldots,a_{m-q}) \not \in E_k^N(C_3B)$$ for all $N\ge N_0$. By the definition of $E_k^N(C_3B)$ and (\ref{coord2}), rescaling by $b_1$, we obtain that for all $\om \in [B^{-1},B]$ there are at least $\lfloor N/k\rfloor$ integers $n\in [1,N]$ for 
which $$\|\om|\zeta^n(v)|_{\vec{s}}\|\ge \rho,$$ hence
$$
\max_{|\om|\in [\Bu^{-1},\Bu]} \prod_{n=1}^{N} (1-c_1\|\om|\zeta^n(v)|_{\vec{s}}\|^2)\le (1-c_1\rho^2)^{\lfloor N/k\rfloor},\ \ \mbox{for all}\ N\ge N_0.
$$
Combined with Proposition~\ref{prop-Dioph0}(ii), this estimate implies 
$$
\sup\left\{ \bigl|S_R^{(x,t)}(\One_{\Xx_a},\om)\bigr|:\ (x,t)\in\Xxi^{\vec{s}},\ |\om|\in [B^{-1},B]\right\}\le C''R^\alpha\ \ \ \mbox{for}\ R\ge R_0(\vec{s}, B),
$$
where
$$
\alpha = 1 + \frac{\log_\theta(1-c_1\rho^2)}{k}\,.
$$
Now the claim of Theorem~\ref{th-holder2} follows from Lemma~\ref{lem-var2}, with $\Om(r) = r^{2-2\alpha}$, and it remains to prove Proposition~\ref{prop-EKvar}.
\end{proof}

\begin{proof}[Proof of Proposition~\ref{prop-EKvar}]
This is proved by a variant of the ``Erd\H{o}s-Kahane argument,'' which was invented to prove power decay of the Fourier transform of almost every Bernoulli convolution (see \cite{Erd}, \cite{Kahane}).  Here we need a generalization which differs in many details, but the basic scheme is the same.
See Appendix (Section 8.1) for a discussion of Bernoulli convolutions.

We can assume that $\Sf$ has no zero eigenvalues, that is, $\theta_m\ne 0$. Indeed, if $\theta_m=0$, we can just ignore the last coordinate and work with the vectors $(a_1,\ldots, a_{m-1})$ in (\ref{eq2}).

A  technical complication is that some of the coordinates of a point in $\Hk^{m-q-1}$ might be zero. In order to address this issue, we
let $\Upsilon$ be the collection of subsets $\Jk\subset \{2,\ldots,q\}$ which satisfy the property
$$
j\in \Jk,\ \ \theta_{j'} = \ov{\theta_j}\ \Rightarrow\ j'\in \Jk.
$$
Then we can write
$$
\Hk^{m-q-1} = \bigcup_{\Jk\in \Upsilon} \bigcup_{\beta>1} \Hk^{m-q-1}(\Jk,\beta),
$$
where
\begin{eqnarray}
\nonumber
\Hk^{m-q-1}(\Jk,\beta)  :=  \bigl\{(a_1,\ldots,a_{m-q})\in \Hk^{m-q-1}: & & 
|a_j|\in [\beta^{-1},\beta],\ \mbox{for all}\ j\in \Jk; \\ 
& & a_j=0,\ \mbox{for all}\ j\in \{2,\ldots,q\}\setminus \Jk\bigr\} \label{gluj}
\end{eqnarray}
(recall that $a_1=1$ for all points in $\Hk^{m-q-1}$). Then we define 
\be \label{def-dumb}
E_k^N(B,\Jk,\beta):= E_k^N(B) \cap \Hk^{m-q-1}(\Jk,\beta)\ \ \mbox{and}\ \ E_k(B,\Jk,\beta):= \bigcap_{N_0=1}^\infty \bigcup_{N=N_0}^\infty E_k^N(B,\Jk,\beta).
\ee
Clearly,
$$
E_k(B) = \bigcup_{\Jk\in \Upsilon} \bigcup_{n=1}^\infty E_k(B,\Jk,1/n), 
$$
hence it is enough to get the desired upper bound (\ref{upper}) for $\dim_H(E_k(B,\Jk,\beta))$ for some fixed $\Jk\ne \es$ and $\beta>0$ (if $\Jk=\es$, the set $E_k(B,\Jk,\beta)$ is trivially a singleton).

To this end, we fix a non-empty $\Jk\in \Upsilon$ and $\beta>0$, and suppose that $(a_1,\ldots,a_{m-q}) \in E_k(B,\Jk,\beta)$. Then $(a_1,\ldots,a_{m-q})$ belongs to $E_k^N(B,\Jk,\beta)$ for infinitely many $N$'s. Fix such an
$N$. By definition, there exist
$\om \in [\Bu^{-1},\Bu]$, and $a_{m-q+1},\ldots,a_m$
for which (\ref{eq2}) holds (recall that we do not exclude the case $q=0$; then no additional $a_j$'s are fixed).  Write
\be \label{eq3}
\om \sum_{j=1}^m a_j \theta_j^n = K_n +\eps_n,\ \ K_n \in \Nat,\ \ |\eps_n|\le 1/2,\ n\ge 1,
\ee
so that $\|\om \sum_{j=1}^m a_j \theta_j^n \|=|\eps_n|$.
Denote
$$
\vec{a} = \left( \begin{array}{c} a_1 \\ \vdots \\ a_m \end{array}\right),\ \ \ \vec{K}_n = \left( \begin{array}{c} K_n \\ \vdots \\ K_{n+m-1} \end{array}\right),\ \ \mbox{and}\ \  
\vec{\eps}_n = \left( \begin{array}{c} \eps_n \\ \vdots \\ \eps_{n+m-1} \end{array}\right);
$$
then  equations (\ref{eq3}) for $n, n+1,\ldots,n+m-1$ combine into 
\be \label{matr1}
\om \left( \begin{array}{ccc} \theta_1^n & \ldots & \theta_m^n \\ \vdots & \ddots & \vdots \\ \theta_1^{n+m-1} & \ldots & \theta_m^{n+m-1} \end{array} \right) \vec{a} = \vec{K}_n + \vec{\eps}_n.
\ee
Let ${\rm Diag}[\theta_j^n]$ be the diagonal matrix with the diagonal entries $\theta_1^n,\ldots,\theta_m^n$, 
then (\ref{matr1}) becomes
$$
\om \,\Theta \cdot {\rm Diag}[\theta_j^n] \,\vec{a} = \vec{K}_n + \vec{\eps}_n,\ \ \ n\ge 1,
$$
where $\Theta$ is the Vandermonde matrix (\ref{Vand}).
The Vandermonde matrix is invertible, since $\theta_j$ are all distinct. Also, all $\theta_j$ were assumed nonzero, hence
\be \label{eq39}
\vec{a} = \om^{-1} {\rm Diag}[\theta_j^{-n}]\,\Theta^{-1}(\vec{K}_n + \vec{\eps}_n),\ \ \ n\ge 1.
\ee
It follows that
\be \label{eq4}
a_j =\om^{-1} \theta_j^{-n} [\Theta^{-1}(\vec{K}_n + \vec{\eps}_n)]_{_{\scriptstyle j}},\ \ \ j=1,\ldots,m,\ \ n\ge 1,
\ee
where $[\cdot]_j$ denotes the $j$-th component of a vector. 
Since $\om\in [B^{-1},B]$, $a_1=1$, and $|a_j|\in [\beta^{-1},\beta]$ for $j\in\Jk$, we have 
\be \label{kapusta0}
B^{-1} \theta_1^n \le |[\Theta^{-1}(\vec{K}_n + \vec{\eps}_n)]_{_{\scriptstyle 1}}| \le B \theta_1^n,\ \ n\ge 1,
\ee
and
\be \label{kapusta}
(\Bu\beta)^{-1} |\theta_j|^n \le |[\Theta^{-1}(\vec{K}_n + \vec{\eps}_n)]_{_{\scriptstyle j}}| \le  \Bu\beta|\theta_j|^n,\ \ \ j\in\ \Jk, \ \ n\ge 1.
\ee
From (\ref{eq4}), recalling that $a_1=1$, we obtain 
\be \label{eq-a}
a_j = \frac{\theta_j^{-n} [\Theta^{-1}(\vec{K}_n + \vec{\eps}_n)]_{_{\scriptstyle j}}}{\theta_1^{-n} [\Theta^{-1}(\vec{K}_n + \vec{\eps}_n)]_{_{\scriptstyle 1}}}\,,
\ee
and we want to show that
$a_j$ is approximately  equal to  $$\frac{\theta_j^{-n} [\Theta^{-1}\vec{K}_n ]_{_{\scriptscriptstyle j}}}{\theta_1^{-n} [\Theta^{-1}\vec{K}_n ]_{_{\scriptscriptstyle 1}}}$$ 
{for} $j \in \Jk$
{and} $n$  {sufficiently large}.
In fact,
\be\label{kapusta2}
\left| \frac{ [\Theta^{-1}(\vec{K}_n + \vec{\eps}_n)]_{_{\scriptstyle j}}}{[\Theta^{-1}(\vec{K}_n + \vec{\eps}_n)]_{_{\scriptstyle 1}}} - 
\frac{ [\Theta^{-1}\vec{K}_n ]_{_{\scriptstyle j}}}{ [\Theta^{-1}\vec{K}_n ]_{_{\scriptstyle 1}}}\right| \le  
\frac{| [\Theta^{-1} \vec{\eps}_n]_{_{\scriptstyle j}}|}{|[\Theta^{-1}(\vec{K}_n + \vec{\eps}_n)]_{_{\scriptstyle 1}}|} + 
\frac{| [\Theta^{-1} \vec{\eps}_n]_{_{\scriptstyle 1}} [\Theta^{-1}\vec{K}_n ]_{_{\scriptstyle j}}|}{|[\Theta^{-1}(\vec{K}_n + \vec{\eps}_n)]_{_{\scriptstyle 1}} [\Theta^{-1}\vec{K}_n ]_{_{\scriptstyle 1}}|}\,.
\ee
Observe that 
$$\|\Theta^{-1}\vec{\eps}_n\|_\infty \le \|\Theta^{-1}\|_\infty \|\vec{\eps}_n\|_\infty\le 
(1/2) \|\Theta^{-1}\|_\infty=:C_\Theta,$$ 
hence (\ref{kapusta0}) and (\ref{kapusta}) yield two-sided estimates of $|[\Theta^{-1}\vec{K}_n]_{_{\scriptstyle j}}|$ for $j\in \{1\}\cup \Jk,\ \ n\ge 1$.
Thus we can continue (\ref{kapusta2}) to obtain for $j\in \Jk$, $n\ge 1$:
$$
\mbox{(\ref{kapusta2})} \ \le\  \frac{C_\Theta}{B^{-1}\theta_1^n} \left( 1 + \frac{B\beta|\theta_j|^n + C_\Theta}{B^{-1}\theta_1^n - C_\Theta}\right) \le 2B C_\Theta \theta_1^{-n}
$$
for $n$ sufficiently large, depending on $\beta,B$, and $\Theta$, since $|\theta_j|<\theta_1$.  Therefore, by (\ref{eq-a}),
\be\label{eq5}
\left|a_j - \frac{\theta_j^{-n} [\Theta^{-1}\vec{K}_n ]_{_{\scriptstyle j}}}{\theta_1^{-n} [\Theta^{-1}\vec{K}_n ]_{_{\scriptstyle 1}}}\right| \le 2B C_\Theta \cdot |\theta_j|^{-n},\ \ \ j\in \Jk,\ \ n\ge n_0(\Theta,\Bu,\beta).
\ee
It is crucial, of course, that $|\theta_j|>1$ for $j\in \Jk \subset \{2,\ldots,m-q\}$. 

The last inequality will be useful a bit later. Now, comparing (\ref{eq39}) with the same equality for $n+1$, we obtain
\be \label{ura}
\vec{K}_{n+1} + \vec{\eps}_{n+1} = \Theta \,{\rm Diag}[\theta_j]\,\Theta^{-1} (\vec{K}_n + \vec{\eps}_n).
\ee

\begin{lemma} \label{lem-step}
Let $\rho$ and $L$ be the constants given by (\ref{def-Lrho}). Consider arbitrary $\om>0$ and $\vec{a} = (a_1,\ldots,a_m) \in \Hk^{m-1}$, and define $K_n,\eps_n$, $n\ge 1$, by the formula (\ref{eq3}).

{\bf (i)} if\ \ $\max\{|\eps_n|,\ldots,|\eps_{n+m}|\} <  \rho$, then $K_{n+m}$ is uniquely determined by
$K_n, K_{n+1}, \ldots, K_{n+m-1}$, independent of $\om$ and $\vec{s}$;

{\bf (ii)} given $K_n, K_{n+1}, \ldots, K_{n+m-1}$, there are at most $L$ possibilities for $K_{n+m}$.
\end{lemma}

\begin{proof}
It follows from (\ref{ura}) that 
$$
\|\vec{K}_{n+1} - \Theta \,{\rm Diag}[\theta_j]\,\Theta^{-1} \vec{K}_n\|_\infty \le
\|\vec{\eps}_{n+1}\|_\infty + \theta_1\|\Theta\|_\infty\|\Theta^{-1}\|_\infty\|\vec{\eps}_n\|_\infty
$$
(we used $\|{\rm Diag}[\theta_j]\|_\infty = \theta_1$ here).
Note that $K_{n+m} = (\vec{K}_{n+1})_{m}$. In part (i) we have \\ $\max\{\|\eps_n\|,\|\eps_{ n+1}\|\}< \rho$, hence
$$
|K_{n+m} - (\Theta \,{\rm Diag}[\theta_j]\,\Theta^{-1} \vec{K}_n)_m| < \rho (1+\theta_1\|\Theta\|_\infty\|\Theta^{-1}\|_\infty) \le 1/2
$$
by (\ref{def-Lrho}), hence $K_{n+m}$ is determined uniquely, being an integer.
In part (ii) we have
$$
|K_{n+m} - (\Theta \,{\rm Diag}[\theta_j]\,\Theta^{-1} \vec{K}_n)_m| \le  (1+\theta_1\|\Theta\|_\infty\|\Theta^{-1}\|_\infty)/2, 
$$
and the number of possible integers $K_{n+m}$ is at most $1+\theta_1\|\Theta\|_\infty\|\Theta^{-1}\|_\infty+1=L$, as desired.
\end{proof}

 Now we conclude the proof of Proposition~\ref{prop-EKvar}. Recall that we are now working with the set $E_k(B,\Jk,\beta)$, see (\ref{def-dumb}), and we estimate its Hausdorff dimension from above by producing efficient covers of the sets $E^N_k(B,\Jk,\beta)$ for sufficiently large $N$. We take an arbitrary point $$(a_1,\ldots,a_{m-q}) \in E^N_k(B,\Jk,\beta)$$ and find the numbers $K_n,\eps_n$ from (\ref{eq3}). The inequality (\ref{eq5}) was proved for $n\ge n_0(\theta,B,\beta)$, and we apply it for $n = N-m+1$. Using that  $$|\theta_{m-q}| = \min_{j\le m-q}|\theta_j|>1,$$ we obtain that $(a_1,\ldots,a_{m-q})$ is contained in the closed $\ell^\infty$ ball of
 radius $2BC_\Theta\cdot |\theta_{m-q}|^{-N+m-1}$, centred at the point
 $$
 (x_1,\ldots,x_{m-q}),\  \mbox{where}\  x_1=1;\ x_j=0,\ j\not\in \Jk;\ \mbox{and}\ x_j = \frac{\theta_j^{-N+m-1} [\Theta^{-1}\vec{K}_{N-m+1} ]_{_{\scriptstyle j}}}{\theta_1^{-N+m-1} [\Theta^{-1}\vec{K}_{N-m+1} ]_{_{\scriptstyle 1}}},\ \ j\in \Jk.
 $$
The number of such balls does not exceed the number of possible vectors $\vec{K}_{N-m+1}$. This, in turn, is bounded above by the
number of possible sequences $K_1,\ldots,K_N$. 
Now we use the crucial assumption (\ref{eq2}) in the definition of the set $E_k^N(B) \supset E_k^N(B,\Jk,\beta)$. The set $$\{n\in [1,N]:\ |\eps_n| \ge \rho\}$$ has cardinality less than $N/k$, and we can enlarge it arbitrarily
to get a set $\Gam \subset [1,N]\cap \Nat$ with $\card(\Gam) = \lceil \frac{N}{k} \rceil$. There are  ${N \choose \lceil N/k \rceil}$ such subsets $\Gam$, and it remains to estimate
the number of possible sequences $K_1,\ldots,K_N$ for a fixed $\Gam$.

Since $(a_1,\ldots,a_m) \in \Fk(\Delta^{m-1})$, it follows from (\ref{def-Fk}) that there exists a constant $C_4>0$, depending only on the substitution matrix and $v$, such that 
$
|a_j| \le C_4,\ j=1,\ldots,m.
$
Further, $|\om| \in [B^{-1},B]$, hence  (\ref{eq3}) implies an upper bound 
$$
|K_n| \le BC_4 m \theta_1^n + 1,\ \ n\ge 1.
$$
Thus, there are at fewer than $C_5$ possibilities for the number of initial parts of the sequence $K_1,\ldots,K_m$, where $C_5 = (BC_4 m \theta_1^m +1)^m$.

Now we fix $\Gam\subset [1,N]\cap \Nat$ and consider those
$(a_1,\ldots,a_{m-q})$ for which
$|\eps_n|< \rho$ for $n\in [1,N]\setminus \Gamma$. 
Once $K_1,\ldots,K_n$ are determined, for $m\le n \le N-1$, we check whether $\{n-m+1,\ldots,n+1\}$ intersects $\Gam$. If it does, there are at most $L$ possibilities for $K_{n+1}$ by Lemma~\ref{lem-step}(ii).
If it does not, then there is only one choice of $K_{n+1}$. It follows that the number of sequences $K_1,\ldots,K_N$ for the given $\Gam$ does not exceed $C_5\cdot L^{(m+1)\card(\Gam)}$. 

Thus, the total number sequences, hence the balls of radius $2BC_\Theta\cdot |\theta_{m-q}|^{-N+m-1}$ needed to cover $E_k^N(B,\Jk,\beta)$ is at most
$$
C_5 {N \choose \lceil N/k \rceil} \cdot L^{(m+1)\lceil N/k\rceil }.
$$ 
Therefore,
$$
\dim_H(E_k(B,\Jk,\beta)) \le\lim_{N\to \infty} \frac{\log\left(C_5 {N \choose \lceil N/k \rceil}\cdot L^{(m+1)\lceil N/k\rceil}\right) }{-\log(2BC_\Theta\cdot |\theta_{m-q}|^{-N+m-1} )} \le \frac{\log[2L^{m+1} k]}{k\log|\theta_{m-q}|}\,.
$$
The proof of Proposition~\ref{prop-EKvar} is concluded, and Theorem~\ref{th-holder2} is now proved completely.
\end{proof}


\section{Self-similar suspension flows}

We continue to study suspension flows $(\Xxi^{\vec{s}},h_t)$ over substitution $\Z$-actions, but now we focus on the special choice of the roof function, which makes the system ``geometrically self-similar'': namely we
assume that $\vec{s}$ is the Perron-Frobenius eigenvector of the transpose substitution matrix $\Sf^t$: $\Sf^t \vec{s} = \theta \vec{s}$. In this case
we have for every word $v$:
$$
|\zeta^n(v)|_{\vec{s}} = \langle \Sf^n \vec{\ell}(v), \vec{s}\rangle = \langle \vec{\ell}(v),(\Sf^t)^n \vec{s}\rangle = \theta^n \langle \vec{\ell}(v),\vec{s}\rangle = \theta^n |v|_{\vec{s}}.
$$
Self-similar suspension flows over substitutions are a special case of self-similar tiling dynamical systems, studied in \cite{SolTil,BuSol} and many other
papers. Much of what we do in this paper can be extended to the tiling setting.
Below we omit the superscript $\vec{s}$ from the notation, and let 
$(\Xxi,h_t)$ be the self-similar substitution suspension flow.

The connection between  substitutions
and self-similar substitution suspension flows can be expressed using the formalism of
Vershik's automorphisms. Recall that a general construction of Vershik \cite{Vershik}
endows an arbitrary ergodic automorphism of a Lebesgue space with a
sequence of Rokhlin towers that intersect in Markovian way; the initial
automorphism is thus represented as a Vershik automorphism of a Markov
compactum. Orbits of our automorphism are identified with leaves of the
tail equivalence relation in the Markov compactum.

In the particular case of substitutions, the construction of
Vershik and Livshits \cite{VerLiv} yields a representation of our initial substitution
automorphism as an ``adic transformation" in a special Markov compactum, the
space of one-sided infinite paths in a fixed finite graph.
Considering now the space of {\it bi-infinite} paths, we arrive at a
natural symbolic coding of the self-similar translation flow, whose orbits
are again identified with leaves of the
tail equivalence relation in the two-sided Markov compactum, see \cite{Bufetov0}.
Self-similar renormalization is effectuated by the Markov shift; note that
in the two-sided case, the invariant measure of the flow
is precisely the measure of maximal entropy for the Markov shift.

In the last section we proved H\"older continuity for almost every suspension flow. However, this says nothing about specific examples. Our next theorem is a quantitative result in the self-similar case, which is weaker 
(logarithmic modulus of continuity instead of the H\"older property), but has the advantage of being specific.

\begin{theorem} \label{th-main2}
Let $\zeta$ be a primitive aperiodic substitution on $\Ak=\{1,\ldots,m\}$, and let $\theta=\theta_1$ be the Perron-Frobenius eigenvalue of the substitution matrix.
Suppose that $\theta$ admits  a Galois conjugate $\theta_2$ outside the unit circle.
Let $(\Xxi,h_t)$ be the self-similar suspension flow over the substitution dynamical system $(X_\zeta, T_\zeta)$, and let $\sig_a$ be the spectral measure corresponding to the cylinder set $[a]$ for $a\in \Ak$.
 Then there exists $\gam>0$ such that for any $\Bu>1$ there exist $C_\Bu$ and $r_0(\Bu)$ such that
 \be \label{eq-main2}
\sig_a([\om-r,\om+r]) \le C_\Bu(\log(1/r))^{-\gam},\ \ \mbox{for all} \ 0 < r \le r_0(\Bu),\ \ a\in \Ak,\  |\om| \in [\Bu^{-1},\Bu].
\ee
\end{theorem}
\begin{remark} The assumption  that $\theta$ admit  a Galois conjugate $\theta_2$ outside the unit circle is equivalent to saying that $\theta$
is not a PV or Salem number.
The definition of PV and Salem numbers is recalled in the Appendix (Section 8.1).
\end{remark}

\begin{remark}
If $\theta$ is a PV number, then the substitution suspension flow has
a dense point spectrum, see  \cite{SolTil,CSa}, so (\ref{eq-main2}) cannot hold. We do not know whether (\ref{eq-main2}) holds for Salem numbers $\theta$.
\end{remark}

Using symbolic coding of translation flows along stable foliations of pseudo-Anosov automorphisms by suspension flows over Vershik's automorphisms ( see Section 1.8.2 in \cite{Bufetov0} and references therein) we obtain the following
\begin{corollary} \label{cor-IET}
Suppose that $g$ is a pseudo-Anosov diffeomorphism on a surface $M$ such that the induced action
$g^*$ on the cohomology group $H^1(M, {\mathbb R})$ has dominant eigenvalue with at least one conjugate outside the unit circle.
Then the conclusion of Theorem \ref{th-main2} 
holds for translation flows along stable/unstable foliations of $g$.
\end{corollary}

The key ingredient of the proof of Theorem \ref{th-main2} is the following proposition, together with Proposition~\ref{prop-Dioph0} and Lemma~\ref{lem-var2}. Recall that the {\em height} of a polynomial is the maximum of absolute values of its coefficients.
 
\begin{prop} \label{prop-alg}
Suppose that $\theta=\theta_1>1$ is an algebraic integer which has at least one conjugate $\theta_2$ satisfying
$|\theta_2|>1$. Then there exist $C>0$ and $\alpha,\beta>0$ such that 
\be\label{eq-alg3}
\exp\Bigl( - \sum_{k=0}^{N-1} \|t\theta^k\|^2 \Bigr) \le \left\{ \begin{array}{rll} C \bigl(\log(1+t)\bigr)^{1/\log\beta} \cdot N^{-\alpha}, &\ \  N\ge 1, &\ \  \mbox{if}\ t\ge 1; \\
                                                                                                                              C N^{-\alpha}, &\ \  N \ge 2\lceil\frac{\log(1/t)}{\log \theta}\rceil, &\ \  \mbox{if}\ t \in (0,1). 
                                                                                                                              \end{array} \right.
\ee
In fact, we can take
$$
\alpha = \frac{(1+sH)^{-2}}{\log\beta}\,,\ \ \ \beta = 1+\lceil s\log\theta/\log|\theta_2|\rceil.
$$
where $H$ is the height of the minimal polynomial of $\theta$ and $s$ is its degree.
\end{prop}

Let us derive Theorem \ref{th-main2} first. Passing to $\zeta^\ell$ if necessary, we can assume that a return word $v$ as in Proposition~\ref{prop-Dioph0} exists and we fix this word for the duration of the proof. Then we are going to apply (\ref{eq-matrest10}) with $\om \in [B^{-1},B]$, keeping in mind that  $|\zeta^k(v)|_{\vec{s}}=\theta^k|v|_{\vec{s}}$ in the self-similar case. We obtain for $(x,t)\in \Xxi$, $\om\in\R$, $R>1$,
and $a\in \Ak$:
$$
|S^{(x,t)}_R(\One_{\Xx_a},\om)|  \le  C' R\exp\Bigl(-c_1\sum_{k=0}^{\lfloor \log_\theta R- C_2\rfloor} \bigl\|\om |v|_{\vec{s}}\cdot\theta^k\bigr\|^2\Bigr) 
                                                          \le  C'_B R (\log_\theta R)^{-c_1\alpha},
$$
using Proposition \ref{prop-alg} in the last inequality. Then Lemma \ref{lem-var2} applies with $\Om(r) = (\log_\theta(1/r))^{-2c_1\alpha}$ since (\ref{eq-SR}) holds with $Y = \Xxi$ and $f=\One_{\Xx_a}$ for
$R_0=R_0(B)$. This yields the
desired estimate for the spectral measure. Now Theorem \ref{th-main2} is proved completely, modulo Proposition \ref{prop-alg}.

\begin{sloppypar}
Interestingly, there is a close relationship between estimates of  expressions like
$\exp( - \sum_{k=0}^{N-1} \|t\theta^k\|^2 )$ and decay estimates of the Fourier transforms of Bernoulli convolutions and other self-similar measures with uniform contraction rates. We discuss Bernoulli convolutions and this connection in the Appendix.
\end{sloppypar}

\begin{proof}[Proof of Proposition~\ref{prop-alg}] The argument uses techniques common in the work of Pisot and Salem and the theory of numbers named after them; it also shares some common features with the proof
of Proposition~\ref{prop-EKvar}.
We will use symbols $\gtrsim$ and $\lesssim$ to indicate inequalities up to a multiplicative constant, depending only on $\theta$. Let $s$ be the degree of $\theta$ and let
$$
q(x) = x^s - b_1 x^{s-1} - \ldots - b_s,\ \ b_j\in \Z,
$$
be the minimal polynomial for $\theta$. First assume that $t\ge 1$  and write
\be \label{eq-diop1}
t\theta^k = p_k + \eps_k,\ \ p_k \in \N, \ \eps_k \in [-1/2,1/2),\ \ \mbox{for}\ k \ge 1.
\ee
Thus $\|t\theta^k\|=|\eps_k|$; note that $p_1 \ge 1$ by the assumption $t\ge 1$. The numbers $p_k$ and $\eps_k$ depend on $t$; we suppress this dependence in the notation, but should keep it in mind. Since the sequence $\{t\theta^k\}_{k\ge 1}$ satisfies the recurrence relation with characteristic polynomial $q(x)$ , we have
$$
p_{k+s} + \eps_{k+s} = b_1 (p_{k+s-1} + \eps_{k+s-1}) + \cdots + b_s (p_k + \eps_k)\ \ \mbox{for}\ k\ge 1,
$$
hence
\be \label{chad1}
p_{k+s} - b_1 p_{k+s-1} - \ldots - b_s p_k = -(\eps_{k+s} - b_1 \eps_{k+s-1} - \ldots - b_s \eps_k).
\ee 
Let $H= \max_{j\le s} |b_j|$ denote the height of the polynomial $q$. 
Since the left-hand side of (\ref{chad1}) is an integer, if
\be \label{chad2}
\max\{|\eps_k|, |\eps_{k+1}|,\ldots, |\eps_{k+s}|\} < (1+sH)^{-1}=: \delta_1,
\ee
then the expression in (\ref{chad1}) is less that one in absolute value, hence zero, and then
$$
\vec{\eps}_{k+1}:=\left( \begin{array}{c} \eps_{k+1} \\ \eps_{k+2} \\ \vdots \\ \eps_{k+s-1} \\ \eps_{k+s} \end{array} \right) = \left( \begin{array}{ccccc}
                     0 & 1 & 0 & \cdots & 0 \\ 0 & 0 & 1 & \cdots & 0  \\ \vdots & \vdots & \vdots & \ddots & \vdots \\ 0 & 0 & 0 & \cdots & 1 \\ b_s & b_{s-1} & b_{s-2} & \cdots & b_1 \end{array} \right)
                     \left( \begin{array}{c} \eps_{k} \\ \eps_{k+1} \\ \vdots \\ \eps_{k+s-2} \\ \eps_{k+s-1} \end{array} \right)=: A \vec{\eps}_k.
$$                   
Note that $A$ is the companion matrix of the polynomial $q(x)$, so its eigenvalues are precisely $\theta$ and its conjugates, and they are simple.
We see that 
\be \label{chad3}
\max\{\|\vec{\eps}_k\|_\infty, \|\vec{\eps}_{k+1}\|_\infty,\ldots,\|\vec{\eps}_{k+n}\|_\infty\} < \delta_1\ \Longrightarrow\ \vec{\eps}_{k+j} = A^j \vec{\eps}_k\ \ \mbox{for}\ j=1,\ldots,n.
\ee
Now we express $\vec{\eps}_k$ as a linear combination of eigenvectors for $A$. Our goal is to show that the coefficient corresponding to $\theta_2$ is not too small in absolute value, so we can estimate
the norms of $\vec{\eps}_{k+j}$ from below, and for $n\sim k$ the implication (\ref{chad3}) will lead to a contradiction.

Let $\theta=\theta_1,\theta_2,\ldots,\theta_s$ be the (real and complex) zeros of $q(x)$. 
Denote by $\{\vec{e}_{\theta_j}\}_{j\le s}$, respectively $\bigl\{\overrightarrow{e^*_{\theta_j}}\bigr\}_{j\le s}$,  the eigenvectors of $A$ and  $A^t$, which can be explicitly written as follows: 
\be \label{eq-diop2}
\vec{e}_{\theta_j} = \left( \begin{array}{c} 1 \\ \theta_j \\ \vdots \\ \theta_j^{s-1}\end{array} \right),\ \ \ \ \overrightarrow{e^*_{\theta_j}} = 
\left( \begin{array}{c} b_s \theta_j^{s-2} \\ b_{s-1} \theta_j^{s-2} + b_s \theta_j^{s-3} \\ \vdots \\ b_2 \theta_j^{s-2} + \cdots + b_{s-1} \theta_j + b_s \\ \theta_j^{s-1}
\end{array} \right).
\ee
The coefficient corresponding to $\theta_2$ in the eigenvector expansion for $\vec{\eps}_k$ equals
$$
a_2 = \frac{\langle\vec{\eps}_k, \overrightarrow{e^*_{\theta_2}}\rangle}{\langle\vec{e}_{\theta_2},\overrightarrow{e^*_{\theta_2}}\rangle}\,.
$$
The denominator in the formula for $a_2$ depends only on $\theta$. As for the numerator, it follows from (\ref{eq-diop1}) that 
$$
\vec{\eps}_k = t\theta^k \vec{e}_{\theta_1} - \vec{p}_k, \ \ \ \mbox{where}\ \ \vec{p}_k = [p_k, p_{k+1},\ldots, p_{k+s-1}]^t.
$$
Thus,
$$
\langle\vec{\eps}_k,\overrightarrow{e^*_{\theta_2}}\rangle= - \langle\vec{p}_k, \overrightarrow{e^*_{\theta_2}}\rangle.
$$
It follows from (\ref{eq-diop2}) that $\langle\vec{p}_k, \overrightarrow{e^*_{\theta_2}}\rangle=Q(\theta_2)$, where $Q$ is a nontrivial polynomial of degree $s-1$, with integer coefficients, of height at most $p_{k+s-1} s H \lesssim t\theta^k$. Since $Q(\theta_2)\ne 0$, applying the classical estimate of Garsia \cite[Lemma 1.51]{Garsia} we obtain
$$
|Q(\theta_2)| \ge \frac{\prod_{|\theta_j|\ne 1, j\ne 2} \bigl||\theta_j|-1\bigr|}{s^{s-2} \bigl(\prod_{|\theta_j|>1,j\ne 2} |\theta_j|\bigr)^s {\rm Height}(Q)^s} \gtrsim t^{-s}\theta^{-ks}.
$$
We have proved that $|a_2|\gtrsim t^{-s}\theta^{-ks}$. Let $|||\cdot|||$ be a norm in $\R^s$ adapted to the eigenvector expansion for $A$ (e.g.\ the sup norm of the vector of coordinates with respect to the basis of eigenvectors given by (\ref{eq-diop2})); then $|||\vec{\eps}_k|||\ge |a_2|$, hence (\ref{chad3}) yields the implication
$$
\max\{\|\vec{\eps}_k\|_\infty, \|\vec{\eps}_{k+1}\|_\infty,\ldots,\|\vec{\eps}_{k+n}\|_\infty\} <\delta_1\ \Longrightarrow\ |||\vec{\eps}_{k+n}|||\ge |\theta_2|^n |a_2|\gtrsim |\theta_2|^n t^{-s} \theta^{-ks}.
$$
Since all norms in $\R^s$ are equivalent, this will be a contradiction with $\|\vec{\eps}_{k+n}\|_\infty \le 1/2$ for $n\ge k\frac{s\log\theta}{\log |\theta_2|}+ \frac{s\log t}{\log|\theta_2|}+K$,  where $K\in \N$ depends only on $\theta$.
Let 
$$
\beta:= \lceil s\log\theta/\log|\theta_2| \rceil + 1.
$$
Then 
$$
k\beta-1\ge k\frac{s\log\theta}{\log |\theta_2|}+\frac{s\log t}{\log|\theta_2|}+K\ \ \mbox{ for }\ \ k\ge k_0:=\lceil s\log t/\log|\theta_2|\rceil+ K+1,
$$
 hence
$$
\forall\,k\ge k_0,\ \ \ \ \max\{\|\vec{\eps}_k\|_\infty, \|\vec{\eps}_{k+1}\|_\infty,\ldots,\|\vec{\eps}_{k\beta-1}\|_\infty\} \ge \delta_1.
$$
This implies that
$$
\exp\Bigl( - \sum_{k=k_0}^{k_0 \beta^n} \|t\theta^k\|^2 \Bigr) \le \exp (-n\delta_1^2),
$$
therefore,
$$
\exp\Bigl( - \sum_{k=0}^N \|t\theta^k\|^2 \Bigr) \lesssim k_0^{1/\log \beta} \exp\bigl(-\frac{\log N}{\log\beta}\cdot \delta_1^2\Bigr) \lesssim  \bigl(\log(1+t)\bigr)^{1/\log\beta} \cdot N^{-\delta_1^2/\log\beta},
$$
which completes the proof of (\ref{eq-alg3}) for $t\ge 1$. Finally, for $t\in (0,1)$ let $j= \lceil \log(t^{-1})/\log\theta\rceil=\min\{j\ge 1:\ t\theta^j\ge 1\}$. Then taking $\tau=t\theta^j$ we
have from the case already proven, for $N \ge 2j$:
$$
\exp\Bigl( - \sum_{k=0}^N \|t\theta^k\|^2 \Bigr) \le \exp\Bigl( - \sum_{k=0}^{N-j} \|\tau\theta^k\|^2 \Bigr)  \lesssim (N-j)^{-\alpha} 
\lesssim (N/2)^{-\alpha},
$$
which yields the desired estimate.
\end{proof}


\section{Spectral measure at zero for self-similar substitution flows}

Here we continue the study of the spectrum for the self-similar substitution suspension flow $(\Xxi,h_t,\wtil{\mu})$ over the substitution $\Z$-action $(X_\zeta,T_\zeta,\mu)$, as in Section 5, but now we focus on 
the behavior of spectral measures $\sig_f$ in the neighborhood of zero. 
Our considerations should be compared to the similar issues for the substitution $\Z$-action itself, discussed in Section 3.1.

Lemma~\ref{lem-var2} and related results in the literature indicate that the behavior of $\sig_f(B(r,0))$ is controlled by the growth of the ergodic integrals
$\int_0^t f\circ h_\tau(x)\,d\tau$. (Of course, we need to take $f$ orthogonal to constants, in order to avoid the trivial zero eigenvalue.) Using available results about asymptotic behavior of such ergodic integrals, we  obtain much more precise results near zero than at other points.

Let $\zeta$ be a primitive aperiodic substitution on the alphabet $\Ak = \{1,\ldots,m\}$, with substitution matrix $\Sf$. 
In this section we assume, in addition, that $\Sf$ has a positive real second eigenvalue greater than one, which dominates the remaining eigenvalues:
\be \label{eigen1}
\theta=\theta_1> \theta_2 > |\theta_3|\ge \cdots,\ \ \ \theta_2>1.
\ee

In order to state our theorem, we need to recall 
results about the asymptotic behavior of ergodic integrals
\be \label{erg-int}
S(f,x,t) = \int_0^t f\circ h_\tau(x)\,d\tau
\ee
for an appropriate class of test functions. In what follows, we use the notation and terminology from \cite{BuSol}, specialized to $d=1$.
Recall that $\Xxi = \{x=(y,u):\ y\in X_\zeta,\ 0 \le u \le s_{y_0}\}$, where $\vec{s} = [s_1,\ldots,s_m]^t$ is the  Perron-Frobenius eigenvector for $\Sf^t$.
We say that a function $f$ on $\Xxi$ is {\em cylindrical} if it is integrable with respect to  $\wtil{\mu}$ and depends only on the ``tile'' containing the origin. More precisely,
there exist functions $\psi_j \in L^1([0,s_j])$ for $j\le m$, such that 
\be \label{eq-cyl}
f(y,t) = \psi_{y_0}(t),\ \ 0 \le t \le s_{y_0}.
\ee

In the self-similar case we  have a ``geometric'' substitution action
$$
Z:\, \Xxi \to \Xxi,
$$
which is hyperbolic (in the sense of Smale spaces) and satisfies the relation
$$
Z\circ h_t = h_{\theta t} \circ Z,
$$
see \cite[2.5]{BuSol}.
Following \cite{Bufetov0,Bufetov1,BuSol}, let $(\Phi_{2,x}^+)_{x\in \Xxi}$  be the H\"older cocycle (or finitely-additive measure) defined on the algebra generated by line segments in $\R$.

For the reader's convenience, we recall the definition of the objects $(\Phi_{2,x}^+)_{x\in \Xxi}$ here, since \cite{Bufetov0,Bufetov1} are concerned
with general translation flows, whereas \cite{BuSol} deals with the more general case of self-similar tilings in $\R^d$, $d\ge 1$. 

As already mentioned at the beginning of Section 4, the elements of $\Xxi$ may be  viewed as tilings of the line $\R$, with tiles of type $j\le m$ being line segments of length $s_j$. The  $\R$-action $h_t$ is then just the translation action $h_t:\,x\mapsto x-t$.
The geometric substitution action $Z$ is the composition of scaling $x\to \theta x$ and subdivision, according to the substitution rule. If we just do
the subdivision, we obtain a sequence of tilings $x^{(-k)}$ for $k\ge 0$, with $x^{(0)} = x$, such that the partition of $\R$ induced by $x^{(-k)}$ refines the
one induced by $x^{(-k+1)}$, for $k\ge 1$. Explicitly, $x^{(-k)} = \theta^{-k} Z^k(x)$. Let $ \vec{e_2^*}=\bigl((\vec{e_2^*})_j\bigr)_{j\le m}$ be an eigenvector of the transpose 
substitution matrix $\Sf^t$ corresponding to $\theta_2$. For an interval tile $I_j-y$ of $x^{(-k)}$ labeled by $j\le m$ we let 
$$
\Phi_{2,x}^+(I_j-y) := \theta_2^{-k} (\vec{e_2^*})_j,
$$
which is then extended by additivity to finite unions and to arbitrary intervals, as a limit over unions of $x^{(-k)}$ tiles approximating them, as $k\to
\infty$. It is
proved in \cite{BuSol} that this is well-defined (under the assumption $|\theta_2|>1$), and we obtain a family of finitely-additive measures $\Phi_{2,x}^+$, $x\in \Xxi$, on the algebra $\FrA$ generated by line segments.
 It has the property
$$\Phi_{2,x}^+(E+t) = \Phi_{2,x-t}(E),\ \ \ x\in \Xxi,\ E\in \FrA.$$
We will also write
$$
\Phi_{2,x}^+(t):= \Phi^+_{2,x}([0,t]).
$$
This function is a cocycle over the $\R$-action $h_t$, since
\begin{eqnarray*}
\Phi_{2,x}^+(s+t) & = & \Phi^+_{2,x}([0,s]) + \Phi^+_{2,x}([s,t]) \\
                            & = & \Phi^+_{2,x}([0,s]) + \Phi^+_{2,x-s}([0,t]) \\
                            & = & \Phi_{2,x}(s) + \Phi^+_{2,h_s(x)}(t).
\end{eqnarray*}
In the next lemma we collect the properties of the cocycle $\Phi_{2,x}^+$.

\begin{lemma} \label{lem-cocycle} Under the standing assumptions ($\theta_2>1$), the cocycle $\Phi_{2,x}^+$ satisfies:

{\bf (i)} $\Phi_{2,Z(x)}^+(\theta t) = \theta_2 \Phi_{2,x}^+(t)$ for $x\in \Xxi$ and $t\in \R$;

{\bf (ii)} $|\Phi_{2,x}^+(t)| \le C_1 \max(1,|t|^\alpha)$, where $\alpha = \log_\theta(\theta_2)$, for some $C_1>0$;

\noindent and is non-degenerate in the following sense:

{\bf (iii)} for a given $x\in \Xxi$, the function $t\mapsto \Phi^+_{2,x}(t)$ is not a.e.\ zero;

{\bf (iv)} for a given $t\in \R_+$, the function $x\mapsto \Phi^+_{2,x}(t)$ is not $\wtil{\mu}$-a.e.\ constant.
\end{lemma}

\begin{proof}
Claims (i) and (ii) are special cases of \cite[Lemma 3.2]{BuSol} and \cite[Lemma 3.3]{BuSol} respectively. Non-degeneracy (iii) and (iv) is proved in the
more general setting of translation flows in \cite[Prop.\,2.26]{Bufetov0} and \cite[Prop.\,2.28]{Bufetov0} respectively. Alternatively, (iii) follows from
the proof of \cite[Lemma 3.5]{BuSol} and (iv) follows from \cite[Section 6.2]{BuSol}.
\end{proof}

Further, let $\Phi_2^{-}$ be the finitely-additive measure on the transversal corresponding to $\theta_2$, which induces a finitely-additive invariant measure $m_{\Phi_2^{-}}$ for the flow $h_t$. There is an explicit formula for $m_{\Phi_2^-}(f)$ for a cylindrical function $f$ in \cite[(39)]{BuSol}:
\be \label{int-cyl}
m_{\Phi_2^{-}}(f) = \sum_{j=1}^m (\vec{e}_2)_j \int_0^{s_j} \psi_j(t)\,dt,
\ee
where $\vec{e}_2$ is the eigenvector of $\Sf$ corresponding to $\theta_2$ 
and $f$ is defined by (\ref{eq-cyl}).

For a cylindrical $f\in L^2(\Xxi,\wtil{\mu})$ with $\int\! f\,d\wtil{\mu}=0$, 
an asymptotic formula  $S(f,x,t)$ is a very special case of results in \cite{Bufetov0}.
It is also a particular case (for $d=1$) of \cite[Theorem 4.3]{BuSol}.  It asserts that there exist $C_1>0$ and $\eps>0$ such that for any cylindrical
function $f$ with zero mean and any $x\in \Xxi$,
\be \label{erg-int2}
S(f,x,t) = \Phi_{2,x}^+(t)\cdot m_{\Phi_2^-}(f) + {\rm \Error}(t),
\ee
where 
\be \label{error}
|{\rm \Error}(t)|\le C_1\max(1,|t|^{\alpha-\eps}),\ \ \ \mbox{with}\ \ \alpha = \log_\theta(\theta_2).
\ee
This should be compared with Dumont-Thomas Theorem~\ref{th-DT} for substitution $\Z$-actions; the corresponding cocycles were further studied in \cite{DKT}.
Now we can state the main result of this section.

\begin{theorem} \label{th-zero}
Let $\zeta$ be a primitive aperiodic substitution on the alphabet $\Ak = \{1,\ldots,m\}$, with substitution matrix $\Sf$ satisfying (\ref{eigen1}). Let $f$ be a cylindrical function with zero mean
$\int f\,d\wtil{\mu}=0$ and $m_{\Phi_2^{-}}(f)\ne 0$. Let $\sig_f$ be the corresponding spectral measure on the line. Then there exists
 a non-trivial positive $\sig$-finite Borel measure $\eta$  on $\R$, such that
\be \label{eq-zero2}
\lim_{N\to \infty}  \frac{\sig_f([-c\theta^{-N}, c\theta^{-N}])}{\theta^{-N(2-2\alpha)}} = \eta([-c,c])\ \ \ \mbox{for all}\ \ c>0\ \mbox{such that}\ \eta(\{c\})=0,
\ee
where
$
\alpha = \log_\theta(\theta_2) \in (0,1).
$
\end{theorem}

\begin{remark}
1. The conditions on the cylindrical functions can be verified directly:
 we have $\int f\,d\wtil{\mu}=\sum_{j=1}^m (\vec{e}_1)_j \int_0^{s_j} \psi_j(t)\,dt$, where $\vec{e}_1$ is the
eigenvector of $\Sf$ corresponding to $\theta=\theta_1$.
The value of
$m_{\Phi_2^{-}}(f)$ is computed in a similar way, see (\ref{int-cyl}). 

2. Note that if $\alpha<1/2$, then our result shows, informally speaking,
that  the   spectral measure ``has a zero of order $1-2\alpha$'' at
zero.

3. An estimate of the spectral measure at zero can be obtained from (\ref{erg-int2}), using Lemma~\ref{lem-var2} and Lemma~\ref{lem-cocycle}(ii). 
Note, however, that this general argument gives much cruder estimates than
those of
Theorem~ \ref{th-zero} and, in particular, {\it does not} give an
asymptotics for the spectral measure.
\end{remark}

\begin{proof}[Proof of Theorem \ref{th-zero}]
Recall that the spectral measure $\sig_f$ on $\R$ is defined by
$$
\int_{-\infty}^\infty e^{2 \pi i\om t}\,d\sig_f(\om) = \langle f\circ h_t, f\rangle\ \ \ \mbox{for}\ t\in \R.
$$
Moreover, there is {\em spectral isomorphism} between $L^2(\R,\sig_f)$ and a closed subspace of $L^2(\Xxi,\wmu)$ which
transforms the unitary group of multiplication by $\{e^{2 \pi i\om t}\}_{t\in \R}$ into the unitary group $g\mapsto
g\circ h_t$. As a consequence, this spectral isomorphism
 maps $e^{2 \pi i\om t}$ (as a function of $\om$) to  $f\circ h_t(x)$ (a function of $x$) for all
$t\in \R$.

Our goal is to analyse the behavior of $\sig_f$ near zero. To this end, we fix a function
$\psi\in\Sk=\Sk(\R)$, the Schwartz class of smooth test functions,
and study the integral
$
 \int |\psi(T\om)|^2\,d\sig_f(\om)\ \ \mbox{for}\ T>0.
$
We obtain, applying the Inverse Fourier Transform (which preserves the class $\Sk$):
$$
\psi(T\om)  = T^{-1}\int_{-\infty}^\infty e^{2 \pi i\om t} \widehat{\psi}(t/T)\,dt.
$$
The latter is
a function of $\om$ in $L^2(d\sig_f)$, which is mapped by the spectral isomorphism to 
\be \label{eq-int1}
T^{-1}\int_{-\infty}^\infty f\circ h_t(x)\cdot \widehat{\psi}(t/T)\,dt,
\ee
a function of $x$ in $L^2(\Xxi,\wmu)$.
Since the spectral isomorphism preserves the norm, we obtain
\be \label{kuka0}
\int |\psi(T\om)|^2\,d\sig_f(\om) = T^{-2} \int_{\Xxi} \Bigl|\int_{-\infty}^\infty f\circ h_t(x)\cdot \widehat{\psi}(t/T)\,dt\Bigr|^2\,d\wmu(x).
\ee
Integration by parts yields, in view of (\ref{erg-int}) and (\ref{erg-int2}):
\begin{eqnarray}
\int_{-\infty}^\infty f\circ h_t(x)\cdot \widehat{\psi}(t/T)\,dt & =& -T^{-1}\int_{-\infty}^\infty S(f,x,t) \bigl({\widehat{\psi}}\bigr)'(t/T)\,dt \nonumber \\
&= & -T^{-1}\int_{-\infty}^\infty\Bigl(\Phi_{2,x}^+(t)\cdot m_{\Phi_2^-}(f) + {\rm \Error}(t)\Bigr)\bigl({\widehat{\psi}}\bigr)'(t/T)\,dt. \label{kuka2}
\end{eqnarray}
 First let us estimate the error term. 
Since $\widehat{\psi}$ is in $\Sk$, we have
$$
|\bigl({\widehat{\psi}}\bigr)'(t)|\le C_{\psi,\alpha}\min(1,|t|^{-\alpha-1}).
$$
Together with
 (\ref{error}), this yields
\begin{eqnarray}
T^{-1}\Bigl|\int_{-\infty}^\infty \Error(t) \bigl({\widehat{\psi}}\bigr)'(t/T)\,dt\Bigr| & \le & \nonumber
2C_1 C_{\psi,\alpha}T^{-1}\Bigl(1 + \int_1^T t^{\alpha-\eps}\,dt + \int_T^\infty t^{\alpha-\eps}(t/T)^{-\alpha-1}\,dt\Bigr) \\ &=& O(T^{\alpha-\eps}).\label{error2}
\end{eqnarray}
For the main term, we assume that
$$
T = \theta^N,\ N\ge 1,
$$
let $t=\theta^N\tau$, and use renormalization (Lemma~\ref{lem-cocycle}(i)):
$$
\Phi_{2,x}^+(t) = \Phi_{2,x}^+(\theta^N \tau) = \theta_2^N \Phi^+_{2,Z^{-N}(x)}(\tau) = T^\alpha 
\Phi^+_{2,Z^{-N}(x)}(\tau).
$$
After the change of variable the main term in (\ref{kuka2}) becomes 
$$
-T^\alpha m_{\Phi_2^-}(f) \int_{\R} \Phi^+_{2,Z^{-N}(x)}(\tau) \bigl({\widehat{\psi}}\bigr)'(\tau)\,d\tau,
$$
We substitute this and (\ref{error2}) into (\ref{kuka2}), which is then substituted into (\ref{kuka0}) to obtain
$$
\int |\psi(T\om)|^2\,d\sig_f(\om) = T^{2\alpha-2} (m_{\Phi_2^-}(f))^2 
\int_{\Xxi} \Bigl|\int_{\R} \Phi^+_{2,x}(\tau)
\bigl({\widehat{\psi}}\bigr)'(\tau)\,d\tau\Bigr|^2 d\wmu(x) + O(T^{2\alpha-2-\eps}),
$$
as $T\to \infty$, where we used that $\wmu$ is $Z$-invariant \cite[Lemma 2.12]{BuSol}.
Finally,  letting $T=\theta^N$, with $N\to \infty$, we obtain
\be \label{kuka3}
\lim_{N\to \infty} \frac{\int |\psi(\theta^N \om)|^2\,d\sig_f(\om)}{(\theta^N)^{2\alpha-2}}= ( m_{\Phi_2^-}(f))^2  \int_{\Xxi} \Bigl|\int_{\R} \Phi^+_{2,x}(\tau)
\bigl({\widehat{\psi}}\bigr)'(\tau)\,d\tau\Bigr|^2\,d\wmu(x).
\ee 
Note that the right-hand is non-trivial (not constant zero), since $\Phi_{2,x}^+$ is non-degenerate and we have freedom in the choice of $\psi$. Observe, moreover, that the entire argument can be repeated for a pair of  real  functions $\psi_1, \psi_2\in \Sk(\R)$, which yields
\begin{eqnarray}
& & Q(\psi_1,\psi_2):=\lim_{N\to \infty} \frac{\int \psi_1(\theta^N  \om)\,\psi_2(\theta^N  \om)\,d\sig_f(\om)}{\theta^{N(2\alpha-2)}} =\nonumber \\[1.3ex]
& & \left(m_{\Phi_2^-}(f)\right)^2  \int_{\Xxi}\left( \int_{\R} \Phi^+_{2,x}(\tau)
\bigl({\widehat{\psi_1}}\bigr)'(\tau)\,d\tau\cdot \int_{\R} \Phi^+_{2,x}(\tau){\bigl({\widehat{\psi_2}}\bigr)'(\tau)}\,d\tau\right)
d\wmu(x). \label{kuka4}
\end{eqnarray}
 We thus have a bilinear continuous functional $Q(\psi_1, \psi_2)$ on $\Sk\times\Sk$ that depends only on the product $\psi_1\psi_2$,
not identically zero and is nonnegative on nonnegative functions.
We now recall the following well-known  fact from the theory of distributions:
Let $D$ be a tempered distribution on the real line such that for any 
 Schwartz function $\varphi\ge 0$ we have $D(\varphi)\geq 0$. Then there exists a $\sigma$-finite positive Radon measure $\eta$ on $\R$ such that 
$$
D(\varphi)=\int_{\R} \varphi\, d\eta.
$$
{\bf Claim.} {\em There exists a $\sigma$-finite positive Radon measure $\eta$ on $\R$ such that for any $\psi_1,\psi_2\in C^\infty_0$ we have
$$
Q(\psi_1,\psi_2) = \int_\R \psi_1\psi_2\,d\eta.
$$
}

In order to prove the claim, observe that for any $\psi_1\in \Sk,\ \psi_1\ge 0$, the functional $\varphi\mapsto Q(\psi_1,\varphi)$ is a tempered distribution which is non-negative on $\varphi\ge 0$,  hence there exists a measure $\eta_{\psi_1}$ such that
$$
Q(\psi_1,\varphi) = \int_\R \varphi\,d\eta_{\psi_1},\ \ \varphi\in \Sk.
$$
Now note that $Q(\psi_1\psi_2,\varphi) = Q(\psi_1,\psi_2\varphi)$, hence
\be \label{temper}
d\eta_{\psi_1\psi_2} = \psi_2\,d\eta_{\psi_1},\ \ \mbox{for}\ \psi_1,\psi_2\in \Sk,\ \psi_1\ge 0,\ \psi_2\ge 0.
\ee
Take any $\psi>0$ in $\Sk$, and let 
$$
d\eta:= \psi^{-1}d\eta_{\psi}
$$
It follows from (\ref{temper}) that $\eta$ does not depend on the choice of $\psi$. Now let $\psi_1,\psi_2\in C^\infty_0$. We have
$$
Q(\psi_1,\psi_2) = Q(\psi, \psi_1\psi_2/\psi) = \int_\R \psi_1\psi_2\psi^{-1}\,d\eta_\psi = \int_\R \psi_1\psi_2\,d\eta,
$$
and the proof of the claim is complete. In the last line we used that $\psi_1\psi_2\psi^{-1}\in C^\infty_0\subset \Sk$.
\qed

\medskip

We  can now conclude that the formula
\begin{equation}\label{kuka5}
\lim_{N\to \infty} \frac{\int \psi_1(\theta^N 
\om)\,\psi_2(\theta^N  \om)\,d\sig_f(\om)}{(\theta^N
)^{2\alpha-2}} =\int_{\mathbb R} \psi_1
\psi_2\,d\eta\end{equation}
holds not just for $\psi_1,\psi_2$ in $C_0^\infty$, but also
if
$\psi_1$ and $\psi_2$ are continuous compactly supported functions or
characteristic functions of  intervals whose endpoints are points of continuity of the measure $\eta$. Indeed, take $a,b\in {\mathbb
R}$ with $\eta(\{a\}) = \eta(\{b\})=0$, let $\psi=\psi_1=\psi_2=\chi_{[a,b]}$, and choose sequences $\psi^{(n,+)},
\psi^{(n,-)}$ of compactly supported
$C^\infty$ functions approximating
$\chi_{[a,b]}$ from above and below, converging to $\chi_{[a,b]}$
pointwise and, for any $\delta>0$, uniformly on the complement to the set
$$
(a-\delta, a+\delta)\cup (b-\delta, b+\delta).
$$
Since $\eta$ does not have atoms at $a$ and $b$, for any  $\varepsilon>0$  there exists
a sufficiently small $\delta$ such that
$$
\eta((a-\delta, a+\delta)\cup (b-\delta, b+\delta))<\varepsilon.
$$
Consequently,
$$
\lim\limits_{n\to\infty}\int_{\mathbb R} \left(\psi^{(n,+)}
\psi^{(n,+)}-\psi^{(n,-)} \psi^{(n,-)}\right)d\eta=0,
$$
and (\ref{kuka5}) holds for $\psi_1=\psi_2=\chi_{[a,b]}$. The case of
compactly supported continuous functions is obtained even easier, by uniform approximation from above and from below.
It remains to substitute the characteristic function $\chi_{[-c,c]}$ for $\psi_1=\psi_2$  to obtain
(\ref{eq-zero2}),
as desired. The theorem is proved completely.
\end{proof}


\section{Dimension of spectral measures}

There are many notions of ``fractal dimension'' for measures; we will be concerned with the {\em Hausdorff dimension} and the
{\em local dimension}.  For a finite Borel measure $\nu$ the Hausdorff dimension is defined by 
$$
\dim_H(\nu) = \inf\{\dim_H(E):\ E\ \mbox{is a Borel set with}\ \nu(E)>0\}.
$$
The {\em lower local dimension} of a measure $\nu$ at $\om$ is defined by $$\underline{d}(\nu,\om) = \liminf_{r\to 0} \frac{\log\nu(B(x,\om))}{\log r}\,.$$
 The upper local dimension is defined similarly, with $\limsup$; the local dimension is said to exist if there is a limit. There is a relation between
these notions, as follows:
$$
\dim_H(\nu) = \sup\{s:\ \und{d}(\nu,\om)\ge s\ \mbox{for $\nu$-almost all}\ \om\}.
$$
For the proof of this, see e.g.\ \cite[Prop.\,10.2]{Falc}. The following is now an immediate consequence of Theorem~\ref{th-holder1}.

\begin{corollary} \label{cor-dim1}
Let $\zeta$ be a primitive aperiodic substitution on $\Ak=\{1,\ldots,m\}$, with substitution matrix $\Sf$.  Suppose that the characteristic polynomial $P_{\Sf}(t)$ is irreducible and the second eigenvalue  satisfies $|\theta_2|>1$. Then there exists $\gam>0$ such that for Lebesgue-almost every suspension flow the spectral measures $\sig_a$, $a\in \Ak$, of the dynamical system $(\Xxi^{\vec{s}},h_t)$ satisfy
$$
\und{d}(\sig_a,\om)\ge \gam > 0\ \ \mbox{for all}\ \om\ne 0.
$$
If $f = \sum_{a\in \Ak} d_a \sig_a$ is orthogonal to constants, then $\dim_H(\sig_f)\ge \gam$.
\end{corollary}

We remark that the same inequality $\dim_H(\sig_{\max})\ge \gam$ can be obtained for the maximal spectral type (in the orthogonal
complement of
constant functions) of a.e.\ suspension flow,
using (\ref{max_type}) and a minor extension of Theorem~\ref{th-holder1}.

We can also estimate the lower local dimensions $\und{d}(\sig_a,\om)$ from below in terms of the top Lyapunov exponent of the matrix
product (\ref{def-Pbi}).

\begin{prop} \label{prop-locdim}
 Let $\zeta$ be a primitive aperiodic substitution on $\Ak = \{1,\ldots,m\}$, $\om \in [0,1)$, and $\M_n(\om)$ are $m\times m$ matrices defined by (\ref{matr}). Suppose that $\theta$ is the Perron-Frobenius eigenvalue of the substitution matrix. Let
 $$
 \alpha_\om = \limsup_{n\to \infty} \|\M_{n-1}(\om) \cdots \M_0(\om)\|^{1/n}.
 $$
 Then 
 \be \label{ki1}
 \limsup_{N\to \infty} \frac{\sup_{x\in X_\zeta} \log|\Phi_a(x[0,N-1],\om)|}{\log N} \le \log_\theta (\alpha_\om) \ \ \ \mbox{for}\ a\in \Ak
 \ee
 and
 \be \label{ki2}
 \underline{d}(\sig_a,\om) \ge 2 - 2\log_\theta(\alpha_\om), \ \ \ \mbox{for}\ a\in \Ak.
 \ee
 \end{prop}
 
 \begin{proof} Note that $\alpha_\om\le \theta$, since $M_j^{|\cdot|} \le \Sf^t$ by the definition (\ref{matr}). If $\alpha_\om=\theta$, then the
inequalities claimed are obvious, so we can assume $\alpha_\om<\theta$.
Fix any $\eps>0$ such that $\alpha_\om+\eps < \theta$. We have $\|\M_{n-1}(\om) \cdots \M_0(\om)\|\le (\alpha_\om+\eps)^n$
for $n\ge n_0$. It follows from (\ref{eq-matr2}) and (\ref{def-Psi})  that 
$$
|\Phi_a(\zeta^n(b),\om)| \le (\alpha_\om+\eps)^n\ \ \mbox{for} \ \ a,b\in \Ak,\ \ n\ge n_0.
$$
Thus we can apply Proposition~\ref{prop-Step1} with $F_\om(n) = C(\frac{\alpha_\om+\eps}{\theta})^n$ for some $C>0$.
Condition (\ref{estim3}) holds with $\theta' = \theta(\alpha_\om+\eps)^{-1}$. We obtain that for $N$ sufficiently large,
$$
|\Phi_a(x[0,N-1],\om)| \le \wtil{C}_1 N^{\log_\theta(\alpha_\om+\eps)},
$$
and since $\eps>0$ can be arbitrarily small, (\ref{ki1}) follows.

Next we prove (\ref{ki2}).
 In view of (\ref{uga1}) and (\ref{eq-GN}), the inequality (\ref{ki1}) implies
 $$
 \limsup_{N\to \infty} \frac{\log G_N(\One_{[a]},\om)}{\log N} \le 2\log_\theta(\alpha_\om)-1.
 $$
 Now Lemma~\ref{lem-easy1} yields 
$$
\liminf_{r\to 0} \frac{\log \sig_a(B(\om,r))}{\log r} \ge 2 - 2\log_\theta(\alpha_\om),
$$
which is (\ref{ki2}), by the definition of lower local dimension.
 \end{proof}
 
 \begin{remark}
 There is a similar statement for suspension flows from Sections 4 and 5. In that setting,
 $$
 \alpha_\om = \limsup_{n\to \infty} \|\M_{n-1}^{\vec{s}}(\om) \cdots \M_0^{\vec{s}}(\om)\|^{1/n},
 $$
 where 
 $\M_j^{\vec{s}}$ are given by (\ref{matrices2}). The matrix product takes a particularly nice form in the self-similar case from Section 5, because then $\M_j^{\vec{s}}(\om) = \M_0^{\vec{s}}(\theta^j \om)$.
 \end{remark}


\section{Appendix}

 This section contains some proofs, as well as definitions and statements, which complement the main body of the paper. 

\begin{proof}[Proof of Lemma~\ref{lem-spec1}]
It is enough to check that the Fourier coefficients of the a.c.\ measures in the right-hand side converge to $\widehat{\sig}_{f,g}(k)$ for all $k\in\Z$. Note that
$$\langle e^{-2\pi i n \om}U^n f, e^{-2\pi i \ell\om}  U^\ell g\rangle = \langle U^{n-\ell} f, g \rangle e^{-2\pi i (n-\ell)\om},$$ and $(e^{-2\pi i n\om} d\om)\,\widehat{\ }\,(-k) = \delta_{k,n}$, and the claim follows easily.
\end{proof}
 
\begin{proof}[Proof of Lemma~\ref{lem-var2}] 
 We have 
 \begin{eqnarray*}
 G_R(f,\om) & = & R^{-1} \left\langle \int_0^R e^{-2\pi i \om y} f\circ h_y \,dy,\int_0^R  e^{-2\pi i \om z} f\circ h_z \,dz\right\rangle \\[1.2ex]
                      & = & R^{-1} \int_{0}^R \int_{0}^R  e^{-2\pi i (y-z)\om} \langle f\circ h_{y-z}, f \rangle \,dy\,dz \\[1.2ex]
                      & = & R^{-1} \int_{0}^R  \int_0^R e^{-2\pi i (y-z)\om} \int_{\R} e^{2 \pi i (y-z)\tau}\,d\sig_f(\tau) \,dy\,dz \\[1.2ex]
                      & = & R^{-1} \int_{\R} \left(\int_{[0,R]^{2}} e^{2\pi i (y-z)(\tau-\om)} dy\,dz\right)\,d \sig_f(\tau) \\
                      & = & \int_\R K_R(\om-\tau)\,d\sig_f(\tau),
 \end{eqnarray*}
 where
 $$
 K_R(y) = R^{-1} \left( \frac{\sin (\pi R y)}{\pi y}\right)^2
 $$
 is the Fej\'er kernel for $\R$. Taking $r = \frac{1}{2R}$, we have $K_R(y)\ge \frac{4R}{\pi^2}$ on $[-r,r]$,
hence 
$
\sig_f([\om-r,\om+r])\le \frac{\pi^2}{4R} G_R(f,\om)$, and (\ref{eq-Hof12})  follows. 
 \end{proof}

\subsection{Classes of algebraic integers and Bernoulli convolutions}

\begin{defi} \label{def-Pisot}
An algebraic integer $\theta>1$ is called a {\em Pisot}, or {\em PV (Pisot-Vijayaraghavan)} number, if all its
Galois conjugates, i.e.\ other zeros of the minimal polynomial for $\theta$, are all inside the unit circle.
An algebraic integer $\theta>1$ is called a {\em Salem} number, if it has no conjugates outside the unit circle and at least one conjugate on the unit circle. (In this case necessarily $\theta$ has even degree $\ge 4$ and
all conjugates except $\theta^{-1}$ are on the unit cirle).
\end{defi}

See \cite{Salem_book} for the basic theory of PV and Salem numbers, as well as their connection to harmonic analysis. These classes of algebraic integers play an important role, not just in the theory of Diophantine approximation and uniform distribution, but also in dynamical systems and fractal geometry, in particular, in the study of Bernoulli convolutions.

\begin{defi}
The (infinite) Bernoulli
convolution measure $\nu_\lam$ with parameter $\lam\in (0,1)$ is the distribution of the random series $\sum_{n=0}^\infty \pm \lam^n$, where the signs are chosen randomly and independently with probabilities $(\half,\half)$.
\end{defi}

Bernoulli convolution measures have been studied extensively since the 1930's, but there are still many difficult open problems; among them an outstanding one is to decide for which $\lam\in (1/2,1)$ the Bernoulli convolution measure $\nu_\lam$ is absolutely continuous. By the ``pure types'' law, $\nu_\lam$ is either singular, or absolutely continuous. The following discussion is not supposed to be comprehensive; many of the statements naturally extend to more general classes of self-similar measures. We will, however, mention the case of ``biased'' Bernoulli convolutions $\nu_\lam^p$, defined analogously to $\nu_\lam$, but with the signs chosen with probabilities $(p,1-p)$, for $p\in (\half,1)$.
The early work on Bernoulli convolutions focused on their Fourier transforms:
$$
\widehat{\nu}_\lam(\xi) = \prod_{n=0}^\infty \cos(2\pi \lam^n \xi).
$$
Erd\H{o}s \cite{Erd0} proved that $\widehat{\nu}_\lam(\xi)\not\to 0$ as $\xi\to\infty$ when $\lam^{-1}$ is a PV number, hence such $\nu_\lam$ are singular. Salem \cite{Salem} proved that this characterizes PV numbers, namely, for all other $\lam$ the Fourier transform vanishes at infinity. However, his proof does not yield any quantitative estimates of the decay of the Fourier transform. Erd\H{o}s
\cite{Erd} proved that for almost every $\lam$ the Fourier transform has some power decay, which he then used to show that $\nu_\lam$ is absolutely continuous for almost every $\lam$ sufficiently close to one. Kahane \cite{Kahane} observed that Erd\H{o}s' proof actually gives an estimate of the dimension of the exceptional set, and in fact the power decay of the Fourier transform (for some power depending on $\lam$) holds for all $\lam$ outside a set of Hausdorff dimension zero. This is what we called the ``Erd\H{o}s-Kahane argument'' in Section 4. The survey \cite{sixty} contains a detailed exposition of these results (Erd\H{o}s' Theorem with Kahane's improvement) with quantitative estimates.  Garsia \cite{Garsia} proved that if $\theta=\lam^{-1}$ is an algebraic integer  whose conjugates are all outside the unit circle and the constant term of the minimal polynomial equals $\pm 2$, then $\nu_\lam$ is absolutely continuous (even with bounded density). As of this writing, this still remains the largest explicitly known class of $\lam$ for which the Bernoulli convolution is absolutely continuous.   It was proved in \cite[Theorem 1.6]{DFW} that the Fourier transform for these ``Garsia'' Bernoulli convolutions has a power decay at infinity. Using different techniques (integration over the parameter and ``transversality method''), Solomyak \cite{SolErd} established absolute continuity of $\nu_\lam$ for a.e.\ $\lam\in (1/2,1)$. For further advances using the transversality method, we  refer the reader to \cite{sixty}. 
Recently  Hochman \cite{Hochman}, using ideas from additive combinatorics, proved that for all $\lam\in (1/2,1)$ outside a set of  Hausdorff dimension zero, and also for all algebraic $\lam\in (1/2,1)$, for which the system has no ``exact overlaps,'' the Hausdorff dimension of $\nu_\lam$ equals one (see Section 7 above for the definition). Having $\dim_H(\nu)=1$ for a measure $\nu$ falls just short of absolute continuity, but even more recently Shmerkin \cite{pablo} combined
Hochman's Theorem with the Erd\H{o}s-Kahane result mentioned above to obtain absolute continuity outside a zero-dimensional set of exceptions.

\medskip

We observe that Proposition~\ref{prop-alg} immediately implies the following

\begin{corollary} \label{cor-BC}
Let $\theta$ be an algebraic integer which has at least one conjugate outside the unit circle, and let $\lam = \theta^{-1}$. Then for any $p\in (0,1)$ there exists $\alpha>0$ such that
$$
\sup_{\xi\in \R} |\widehat{\nu}^p_\lam(\xi)| (\log(2+|\xi|))^\alpha < \infty.
$$
\end{corollary}

\begin{proof}
Let $\xi \in [\theta^N,\theta^{N+1}]$, and set $\zeta = \theta^{-N}\xi$. Then 
\begin{eqnarray*}
|\widehat{\nu}_\lam^p(\xi) | & = & \prod_{n=0}^\infty \Bigl| p e^{-2\pi i \lam^n \xi} + (1-p) e^{2\pi i \lam^n \xi}\Bigr| \\
                                                 & \le  & \prod_{n=0}^N \Bigl| p +(1-p)e^{4\pi i \theta^n \zeta}\Bigr| \\
                                                 & \le & \prod_{n=0}^N \Bigl( 1 - {\frac{1-p}{2}}\|2\theta^n \zeta\|^2\Bigr),
\end{eqnarray*}
using (\ref{lem-elem1}),
and Proposition \ref{prop-alg} yields the result.
\end{proof}

\begin{remark}
1. R. Kershner \cite{Kersh} obtained a similar result for rationals $\lam = a/b$, with $1<a<b$; see also \cite{Dai}.

2.  Corollary~\ref{cor-BC} extends to the class of self-similar measures of the form $\nu = \sum_{j=1}^m p_j (\nu\circ f_j^{-1})$, where
$f_j(x) = \lam x + a_j, \ a_j \in \Z$, and $(p_1,\ldots,p_m)$ is a probability vector (the Bernoulli convolution $\nu_\lam^p$ is a special case with $m=2$, $a_1=-1, a_2=1$).

3. If $\theta>1$ is the Perron-Frobenius eigenvalue of a substitution matrix, and a biased Bernoulli convolution $\nu_\lam^p$, with $\lam = \theta^{-1}$ and $p\ne \half$, has a power decay of the Fourier transform with exponent $\gamma$, then $\sup_{t\in \R} \exp(N\gam - \sum_{k=0}^{N-1} \|t\theta^k\|^2 )<\infty$
and as a consequence, spectral measures of the self-similar suspension flow are H\"older continuous. (The ``bias'' is important here, because the factors $p e^{-2\pi i \lam^n \xi} + (1-p) e^{2\pi i \lam^n \xi}$ are bounded away from zero in absolute value.)
However,  power decay of the Fourier transform is not known for any biased Bernoulli convolution; the result of \cite[Theorem 1.6]{DFW} covers the unbiased case only.
\end{remark}


\end{document}